\documentclass[a4paper,11pt]{article}
\usepackage[english]{babel}

\usepackage{amsxtra}
\usepackage{amssymb}
\usepackage{amsmath}
\usepackage{amsthm}
\usepackage{color}

\usepackage{hyperref}

\newtheorem{prop}{Proposition}[section]
\newtheorem{assumption}{Assumption}[section]

\newtheorem{lemma}{Lemma}[section]
\newtheorem{theo}{Theorem}[section]

\newtheorem{remarka}{Remark}[section]

\numberwithin{equation}{section}

\theoremstyle{remark}
\newtheorem{rmq}[prop]{Remark}

\newcommand{\D}{\Delta}

\newcommand{\n}{\nabla}

\newcommand{\p}{\partial}

\newcommand{\h}{\hookrightarrow}

\newcommand{\ort}{\perp}

\newcommand{\la}{\langle}
\newcommand{\ra}{\rangle}
\newcommand{\ct}{\langle t\rangle}
\newcommand{\cs}{\langle s\rangle}

\newcommand{\e}{\epsilon}
\newcommand{\ve}{\varepsilon}

\newcommand{\R}{\mathbb{R}}
\newcommand{\Z}{\mathbb{Z}}

\newcommand{\va}{\varphi}

\title{Global well-posedness of the Euler-Korteweg system for small irrotational 
data}
\author{Corentin Audiard
\thanks{Sorbonne Universit\'es, UPMC Univ Paris 06, UMR 7598, 
Laboratoire Jacques-Louis Lions, F-75005, Paris, France}
\thanks{CNRS, UMR 7598, Laboratoire Jacques-Louis Lions, F-75005, Paris, France}
and Boris Haspot  \thanks{Universit\'e Paris Dauphine, PSL Research University, Ceremade, Umr Cnrs 7534, Place du Mar\' echal De Lattre De Tassigny 75775 Paris cedex 16 (France), haspot@ceremade.dauphine.fr}}
\date{}
\begin{document}
\maketitle
\begin{abstract}
The Euler-Korteweg equations are a modification of the Euler equations that takes into account
capillary effects. In the general case they form a quasi-linear system that can be recast as 
a degenerate Schr\"odinger type equation. Local well-posedness (in subcritical Sobolev spaces) 
was obtained by Benzoni-Danchin-Descombes in any space dimension, however, except in some special 
case (semi-linear with particular pressure) no global well-posedness is known. We prove here 
that under a natural stability condition on the pressure, global well-posedness holds in dimension 
$d\geq 3$ for small irrotational initial data. The proof is based on a modified energy 
estimate, standard dispersive properties if $d\geq 5$, and a careful study of the nonlinear 
structure of the quadratic terms in dimension $3$ and $4$ involving the theory of space time resonance.
\end{abstract}

\renewcommand{\abstractname}{R\'esum\'e}
\begin{abstract}
Les \'equations d'Euler-Korteweg sont une modification des \'equations d'Euler prenant en 
compte l'effet de la capillarit\'e. Dans le cas g\'en\'eral elles forment un syst\`eme 
quasi-lin\'eaire qui peut se reformuler comme une \'equation de Schr\"odinger d\'eg\'en\'er\'ee.
L'existence locale de solutions fortes a \'et\'e obtenue par Benzoni-Danchin-Descombes en toute
dimension, mais sauf cas tr\`es particuliers il n'existe pas de r\'esultat d'existence 
globale. En dimension au moins $3$, et sous une condition naturelle de stabilit\'e sur la pression 
on prouve que pour toute 
donn\'ee initiale irrotationnelle petite, la solution est globale. La preuve s'appuie sur une 
estimation d'\'energie modifi\'ee. En dimension 
au moins $5$ les propri\'et\'es standard de dispersion suffisent pour conclure tandis que les 
dimensions $3$ et $4$ requi\`erent une \'etude pr\'ecise de la structure des nonlin\'earit\'es 
quadratiques pour utiliser la m\'ethode des r\'esonances temps espaces.
\end{abstract}
\tableofcontents

\section{Introduction}
The compressible Euler-Korteweg equations read 
\begin{equation}
\label{EK}
\left\{
\begin{array}{ll}
\p_{t}\rho+{\rm div}(\rho u)=0,\ (x,t)\in \R^d\times I\\
\p_{t}u+u\cdot \n u+\nabla g(\rho)=\n \bigg(K(\rho)\Delta \rho +\frac{1}{2}K'(\rho)|\nabla \rho|^2
\bigg),\ (x,t)\in \R^d\times I\\
(\rho,u)|_{t=0}=(\rho_{0},u_{0}),\ x\in \R^d.
\end{array}\right.
\end{equation}
Here $\rho$ is the density of the fluid, $u$ the velocity, $g$ the bulk chemical potential, related to 
the pressure by $p'(\rho)=\rho g'(\rho)$. $K(\rho)>0$ corresponds to the capillary coefficient. 
On the left hand side we recover the Euler equations, while the right hand side of the second 
equation contains the so called Korteweg tensor, which is intended to take into account capillary 
effects and models in particular the behavior at the interfaces of a liquid-vapor mixture. The system arises 
in various settings: the case $K(\rho)=\kappa/\rho$ corresponds to the 
so-called equations of quantum hydrodynamics (which are formally equivalent to the Gross-Pitaevskii 
equation through the Madelung transform, on this topic see the survey of Carles et al \cite{CDS}).\\
As we will see, in the irrotational case the system can be reformulated as a 
quasilinear Schr\"odinger equation, this is in sharp contrast with the  non homogeneous 
incompressible case where the system is hyperbolic (see 
\cite{Feireisl}).
For a general $K(\rho)$, local well-posedness was proved in \cite{Benzoni1}. Moreover \eqref{EK}
has a rich structure with special solutions such as planar traveling waves, 
namely solutions that only depend on $y=t-x\cdot \xi$, $\xi\in \R^d$, with possibly 
$\lim_{\infty}\rho(y)\neq \lim_{-\infty} \rho(y)$.
The orbital stability and instability of such solutions has been largely studied over the last 
ten years (see \cite{BDD2} and the review article of Benzoni-Gavage \cite{Benzoni5}).
The existence and non uniqueness of global non dissipative weak solutions \footnote{These global weak solution do not verify the energy inequality} in the spirit of De Lellis-Szekelehidi\cite{DelSz})
was tackled by Donatelli et al \cite{Donatellietcie}, while weak-strong uniqueness has been 
very recently studied by Giesselman et al \cite{GLT}.\\
Our article deals with a complementary issue, namely the global well-posedness and asymptotically 
linear behaviour of small smooth solutions near the constant state $(\rho,u)=(\overline{\rho},0)$.
To our knowledge we obtain here the 
first global well-posedness result for \eqref{EK} in the case of a general pressure and capillary 
coefficient. This is in strong contrast with the existence of infinitely many \emph{weak} 
solutions from \cite{Donatellietcie}.
\\
A precise statement of our results is provided in theorems \ref{theolarged},\ref{theod4} of
section \ref{results}, but first we will briefly 
discuss the state of well-posedness theory, the structure of the equation, and the tools available 
to tackle the problem. 
Let us start with the local well-posedness result from \cite{Benzoni1}.
\begin{theo}
For $d\geq 1$, let $(\overline{\rho},\overline{u})$ be a smooth solution whose 
derivatives decay rapidly at infinity, $s>1+d/2$. Then for $(\rho_0,u_0)\in (\overline{\rho},\overline{u})
+H^{s+1}(\R^d)\times H^s(\R^d)$, 
$\rho_0$ bounded away from $0$, there exists $T>0$ and a unique solution $(\rho,u)$ of 
$(\ref{EK})$ such that $(\rho-\overline{\rho},u-\overline{u})$ belongs to 
$C([0,T],H^{s+1}\times H^s)\cap C^1([0,T],H^{s-1}\times H^{s-2})$ and $\rho$ remains bounded away 
from $0$ on $[0,T]\times \R^d$.
\end{theo}
\noindent We point out that \cite{Benzoni1} includes local well-posedness results for 
nonlocalized initial data (e.g. theorem $6.1$). The authors also obtained several blow-up 
criterions. In the 
irrotational case it reads:
\subparagraph{Blow-up criterion:} \label{blowcriter} for $s>1+d/2$, $(\rho,u)$ solution on 
$[0,T)\times \R^d$ of $(\ref{EK})$, the solution can be continued beyond $T$ provided 
\begin{enumerate}
 \item $\rho([0,T)\times \R^d)\subset J\subset \R^{+*}$, $J$ compact and $K$ is smooth 
 on a neighbourhood of $J$.
 \item $\int_0^T(\|\Delta \rho(t)\|_\infty+\|\text{div}u(t)\|_\infty )dt<\infty$.
\end{enumerate}

These results relied on energy estimates for an extended system that we write now. 
If $\mathcal{L}$ is a 
primitive of $\sqrt{K/\rho}$, setting $L=\mathcal{L}(\rho)$, $w=\sqrt{K/\rho}
\n \rho=\n L$, $a=\sqrt{\rho K(\rho)}$, from basic computations we verify (see \cite{Benzoni1}) that the equations on 
$(L,u,w)$ are
\begin{equation*}
\left\{
\begin{array}{ll}
\p_tL+u\cdot \nabla L+a\text{div}u=0,\\
\partial_tu+u\cdot \nabla u-w\cdot \n w-\n (a\text{div}w)=-\n g,\\
\partial_tw+\n (u\cdot w)+\n (a\text{div}w)=0,
\end{array}
\right.
\end{equation*}
or equivalently for $z=u+iw$
\begin{equation}\label{ES}
\left\{
\begin{array}{ll}
\p_tL+u\cdot \nabla L+a\text{div}u=0,\\
\partial_t z+u\cdot \n z+i(\n z)\cdot w+i\n(a\text{div}z)=\n \widetilde{g}(L).
\end{array}
\right.
\end{equation}
Here we set $\widetilde{a}(L)=a\circ \mathcal{L}^{-1}(L),\ \widetilde{g}(L)=g\circ 
\mathcal{L}^{-1}(L)$ which are well-defined since $\sqrt{K/\rho}>0$ thus $\mathcal{L}$ is 
invertible. \\
This change of unknown clarifies the underlying dispersive structure of the model as the second 
equation is a quasi-linear degenerate Schr\"odinger equation. It should be pointed out 
however that the local existence results of \cite{Benzoni1} relied on $H^s$ energy estimates 
rather than dispersive estimates. On the other hand, we constructed recently 
in \cite{AudHasp} global 
small solutions to $(\ref{EK})$ for $d\geq 3$ when the underlying system is 
\emph{semi-linear}, that is $K(\rho)=\kappa/\rho$ with $\kappa$ a positive constant and for 
$g(\rho)=\rho-1$. This case corresponds to the equations of quantum hydrodynamics. The 
construction relied on the so-called Madelung transform, which establishes a formal 
correspondance between these equations and the Gross-Pitaevskii equation, and recent results 
on scattering for the Gross-Pitaevskii equation \cite{GNT1}\cite{GNT3}. Let us recall for completeness that 
$1+\psi$ is a solution of the Gross-Pitaevskii equation if $\psi$ satisfies 
\begin{equation}\label{GP}
i\partial_t\psi+\Delta \psi-2\text{Re}(\psi)=\psi^2+2|\psi|^2+|\psi|^2\psi.
\end{equation}
For the construction of global weak solutions (no uniqueness, but no smallness assumptions)
we refer also to the work of Antonelli-Marcati \cite{AntMarc2,AntMarc}.\\
In this article we consider perturbations of the constant state $\rho=\rho_c,\ u=0$ for a general 
capillary coefficient $K(\rho)$ that we only suppose smooth and positive on an interval 
containing $\rho_c$. In order to exploit the dispersive nature of 
the equation we need to work with irrotational data $u=\n \phi$ so that $(\ref{ES})$ reduces 
to the following system (where $L_c=\mathcal{L}(\rho_c)$ which has obviously similarities with $(\ref{GP})$ (more details are provided in sections 
$3$ and $4$):
\begin{equation}\label{eqcanonique}
\left\{
\begin{array}{ll}
\partial_t\phi-\Delta (L-L_c)+\widetilde{g}'(L_c)(L-L_c)=\mathcal{N}_1(\phi,L),\\
\partial_t(L-L_c)+\Delta \phi=\mathcal{N}_2(\phi,L)
\end{array}\right.
\end{equation}
The sytem satisfies the dispersion relation 
$\tau^2=|\xi|^2(\widetilde{g}'(L_c)+|\xi|^2)$, and the $\mathcal{N}_j$ are at least quadratic 
nonlinearities that depend on $L,\phi$ and their derivatives (the system is thus quasi-linear). We 
also point out that the stability condition $\widetilde{g}'(L_c)\geq 0$ 
is necessary in order to ensure that the solutions in $\tau$ of the dispersion relation are real.\\
The existence of global small solutions for nonlinear dispersive equations is a rather classical 
topic which is impossible by far to describe exhaustively in this introduction. We shall 
yet underline the main ideas that are important for our work here.
\paragraph{Dispersive estimates}
For the Schr\"odinger equation, two key tools are the dispersive estimate
\begin{equation}
 \|e^{it\Delta}\psi_0\|_{L^q(\R^d)}\lesssim\frac{\|\psi_0\|_{L^2}}{t^{d(1/2-1/q)}},
 \label{adisp1}
\end{equation}
and the Strichartz estimates 
\begin{eqnarray}
 \|e^{it\Delta}\psi_0\|_{L^p(\R,L^q(\R^d))}\lesssim \|\psi_0\|_{L^2},\ \frac{2}{p}+\frac{d}{q}
 =\frac{d}{2},\\
 \|\int_0^te^{i(t-s)\Delta}f(s)ds\|_{L^p(\R,L^q(\R^d))}\lesssim \|f\|_{L^{p'_1}(\R,L^{q'_1}(\R^d)},
 \
\frac{2}{p_1}+\frac{d}{q_1}=\frac{d}{2}.
 \end{eqnarray}
Both indicate decay of the solution for long time in $L^p(L^q)$ spaces, it is of course of interest when we wish to prove the existence of global strong solution since it generally require some damping behavior for long time.  Due to the pressure term the linear structure of our system is actually closer to the one of the 
Gross-Pitaevskii equation (see (\ref{GP})), but the estimates are essentially the same as 
for the Schr\"odinger 
equation. Local smoothing is also an interesting feature of Schr\"odinger equations, in particular 
for the study of quasilinear systems. A result in this direction was obtained by the first author 
in \cite{Audiard4} but we will not need it here. The main task of our proof will consist in proving dispersive estimates 
of the type (\ref{adisp1}) for long time, it is related to the notion of scattering for the solution of the 
dispersive equations. Let us recall now some classical result on the theory of the scattering for the Schr\" odinger 
equations and the Gross Pitaevskii equation.
\paragraph{Scattering} Let us consider the following nonlinear Schr\"odinger 
equation 
\begin{equation*}
i\p_t \psi+\Delta \psi=\mathcal{N}(\psi).
\end{equation*}
Due to the dispersion, when the nonlinearity vanishes at a sufficient order 
at $0$ and the initial data is sufficiently small and localized, it is possible to prove that the 
solution is global and 
the integral $\int e^{-is\Delta}\mathcal{N}(\psi(s))ds$ converges in $L^2(\R^d)$, so that there 
exists $\psi_+\in L^2(\R^d)$ such that 
\begin{equation*}
\|\psi(t)-e^{it\Delta}\psi_+\|_{L^2}\longrightarrow_{t\rightarrow \infty}0.
\end{equation*}
In this case, it is said that the solution is asymptotically linear, or \emph{scatters} to $\psi_+$.\\
In the case where $\mathcal{N}$ is a general power-like non-linearity, we can cite the seminal work
of Strauss \cite{Strauss}. More precisely if $\mathcal{N}(a)=O_0(|a|^p)$, 
global well-posedness for small data in $H^1$ is merely a consequence of Strichartz estimates provided $p$ is 
larger than the so-called Strauss exponent 
\begin{equation}\label{expSL}
p_S(d)=\frac{\sqrt{d^2+12d+4}+d+2}{2d}. 
\end{equation}
For example scattering for quadratic nonlinearities (indepently of their structure 
$\phi^2, \overline{\phi}^2,|\phi|^2$...) can be obtained for $d\geq 4$, indeed $p_S(3)=2$. 
The case $p\leq p_S$ is much harder and is discussed later.

\paragraph{Mixing energy estimates and dispersive estimates}
If $\mathcal{N}$ depends on derivatives of $\phi$, due to the loss of derivatives the situation is 
quite different and it is important to take more precisely  into account  the structure of the system. In particular it is possible in some case to exhibit energy estimates which often lead after a Gronwall lemma to the following situation: 
\begin{equation*}
\forall\,N\in \mathbb{N},\ 
\|\phi(t)\|_{H^N}\leq \|\phi_0\|_{H^N}\text{exp}\bigg(C_N\int_0^t\|\phi(s)\|_{W^{k,\infty}}^{p-1}
ds\bigg),\ k\text{ ``small'' and independent on $N$}.
\end{equation*}
A natural idea consists in mixing energy estimates in the $H^N$ norm, $N$ ``large'', with dispersive 
estimates : if one obtains 
\begin{equation*}
\bigg\|\int_0^t e^{i(t-s)\Delta}\mathcal{N}ds\bigg\|_{W^{k,\infty}}\lesssim 
\frac{\|\psi\|_{H^N\cap W^{k,\infty}}^p}{t^\alpha},\ \alpha(p-1)>1,
\end{equation*}
then setting $\|\psi\|_{X_T}=\sup_{[0,T]}\|\psi(t)\|_{H^N}+t^\alpha\|\psi(t)\|_{W^{k,\infty}}$ 
the energy estimate yields for small data
\begin{equation*}
\|\psi\|_{X_T}\lesssim \|\psi_0\|_{H^N}\text{exp}(C\|\psi\|_X^{p-1})+
\|\psi\|_{X_T}^p+\varepsilon,
\end{equation*}
so that $\|\psi\|_{X_T}$ must remain small uniformly in $T$.
This strategy seems to have 
been initiated independently by Klainerman and Ponce \cite{KlaiPonce} and Shatah \cite{Shatah}.
If the energy estimate is true, this method works ``straightforwardly'' and gives global 
well-posedness for small initial data (this is the approach from 
section $4$) if
\begin{equation}\label{expQL}
p>\widetilde{p}(d)=\frac{\sqrt{2d+1}+d+1}{d}>p_S(d).
\end{equation}
Again, there is a critical dimension: $\widetilde{p}(4)=2$, thus any quadratic nonlinearity 
can be handled with this method if $d\geq 5$.
\paragraph{Normal forms, space-time resonances} When $p\leq p_S$ (semi-linear case) or 
$\widetilde{p}$ (quasi-linear case), the 
strategies above can not be directly applied, and one has to look more closely at the structure 
of the nonlinearity. For the Schr\"odinger equation, one of the earliest result in this direction 
was due to Cohn \cite{Cohn} who proved (extending Shatah's method of normal forms \cite{Shatah2}) 
the global well-posedness in dimension $2$ of 
\begin{equation}\label{eqcohn}
 \begin{array}{ll}
i\p_t \psi+\Delta \psi=i\n \overline{\psi}\cdot \n\overline{\psi}.
 \end{array}
\end{equation}
The by now standard strategy of proof was to use a normal form that transformed the quadratic 
nonlinearity 
into a cubic one, and since $3>\widetilde{p}(2)\simeq 2.6$ the new equation could be treated with the arguments from 
\cite{KlaiPonce}.
In dimension $3$, similar results (with very different proofs using vector fields method 
and time non resonance) were then obtained for the 
nonlinearities $\psi^2$ and $\overline{\psi}^2$ by Hayashi, 
Nakao and Naumkin \cite{HaNau} (it is important to observe that the quadratic nonlinearity is critical in terms of Strauss exponent for the semi-linear case when $d=3$). The existence of global solutions for the nonlinearity $|\psi|^2$ 
is however still open (indeed it corresponds to a nonlinearity where the set of time and space non resonance is not empty, we will give more explanations below on this phenomenon) .\\
More recently, Germain-Masmoudi-Shatah \cite{GMS}\cite{GMS2}\cite{GMS2D} and 
Gustafson-Nakanishi-Tsai \cite{GNT2}\cite{GNT3} shed a new light on such 
issues with the concept of space-time resonances. To describe it, let us rewrite the Duhamel 
formula for the profile of the solution $f=e^{-it\Delta}\psi$, in the case \eqref{eqcohn}:
\begin{equation}\label{duhaprofile}
f=\psi_0+\int_0^te^{-is\Delta}\mathcal{N}(e^{is\Delta }\psi)ds
\Leftrightarrow 
\widehat{f}=\widehat{\psi_0}+\int_0^t\int_{\R^d}e^{is(|\xi|^2+|\eta|^2+|\xi-\eta|^2)}
\eta\cdot (\xi-\eta)\widehat{\overline{f}}(\eta)\widehat{\overline{f}}(\xi-\eta)d\eta ds
\end{equation}
In order to take advantage of the non cancellation of $\Omega(\xi,\eta)=
|\xi|^2+|\eta|^2+|\xi-\eta|^2$ one might 
integrate by part in time, and from the identity $\partial_tf=-ie^{-it\Delta}\mathcal{N}(\psi)$,
we see that this procedure effectively replaces the quadratic nonlinearity by a cubic one, ie acts 
as a normal form.\\
On the other hand, if $\mathcal{N}(\psi)=\psi^2$ the phase becomes $\Omega(\xi,\eta)=
|\xi|^2-|\eta|^2-|\xi-\eta|^2$, which cancels on a large set, namely the ``time resonant set''
\begin{equation}\label{timeR}
 \mathcal{T}=\{(\xi,\eta):\ \Omega(\xi,\eta)=0\}=\{\eta\perp\ \xi-\eta\}.
\end{equation}
The remedy is to use an integration by part in the $\eta$ variable using
$e^{is\Omega}=\frac{\nabla_{\eta}\Omega}{is |\nabla_\eta\Omega|^2}\nabla_\eta(e^{is\Omega})$, it
does not improve the nonlinearity, however we can observe a gain of time decay in $1/s$. This 
justifies to define the ``space resonant set'' as 
\begin{equation}\label{spaceR}
 \mathcal{S}=\{(\xi,\eta):\ \nabla_\eta\Omega(\xi,\eta)=0\}=\{\eta=-\xi-\eta\},
\end{equation}
as well as the space-time resonant set 
\begin{equation}\label{totalR}
\mathcal{R}=\mathcal{S}\cap \mathcal{T}=\{(\xi,\eta):\ \Omega(\xi,\eta)=0,\ \nabla_\eta\Omega
(\xi,\eta)=0\}.
\end{equation}
For $\mathcal{N}(\psi)=\psi^2$, we simply have $\mathcal{R}=\{\xi=\eta=0\}$; using the previous strategy
Germain et al \cite{GMS} obtained global well-posedness for the quadratic Schr\"odinger equation. \\
Finally, for $\mathcal{N}(\psi)=|\psi|^2$ similar computations lead to $\mathcal{R}=\{\xi=0\}$, 
the ``large'' size of this set might explain why this nonlinearity is particularly difficult to 
handle.
\paragraph{Smooth and non smooth multipliers} The method of space-time resonances in the case 
$(\n \overline{\phi})^2$ is particularly simple because after the time integration by part, the 
Fourier transform of the nonlinearity simply becomes 
$$
\frac{\eta\cdot (\xi-\eta)}{|\xi|^2+|\eta|^2+|\xi-\eta|^2}\p_s\widehat{\overline{\psi}}(\eta)
\widehat{\overline{\psi}}(\xi-\eta),
$$
where the multiplier $\frac{\eta\cdot (\xi-\eta)}{|\xi|^2+|\eta|^2+|\xi-\eta|^2}$ is of 
Coifman-Meyer type, thus in term of product laws it is just a cubic nonlinearity. We might naively 
observe that this is due to the fact that 
$\eta\cdot (\xi-\eta)$ cancels on the resonant set $\xi=\eta=0$. Thus one might wonder what 
happens in the general case if the nonlinearity writes as a bilinear Fourier multiplier whose symbol 
cancels on $\mathcal{R}$. In \cite{GMS2D}, the authors treated the nonlinear Schr\"odinger 
equation for $d=2$ by assuming that the nonlinearity 
is of type $B[\psi,\psi]$ or $B[\overline{\psi},\overline{\psi}]$, with $B$ a bilinear Fourier 
multiplier whose symbol is linear at $|(\xi,\eta)|\leq 1$ (and thus cancels on $\mathcal{R}$). 
Concerning the Gross-Pitaevskii equation (\ref{GP}), the nonlinear terms include the worst one 
$|\psi|^2$ but Gustafson et al \cite{GNT3} 
managed to prove global existence and scattering in dimension $3$, one of the important ideas of 
their proof was a change of unknown $\psi\mapsto Z$ (or normal form) that replaced the nonlinearity 
$|\psi|^2$ by $\sqrt{-\Delta/(2-\Delta)}|Z|^2$ which compensates the resonances at $\xi=0$.
To some extent, this is also a strategy that we will follow here.\\
Finally, let us point out that the method of space-time resonances proved remarkably efficient for 
the water wave equation \cite{GMS2} partially because the group velocity 
$|\xi|^{-1/2}/2$ is large near $\xi=0$, while it might not be the 
most suited for the Schr\"odinger equation whose group velocity $2\xi$ cancels at $\xi=0$. The 
method of vector fields is an interesting 
alternative, and this approach was later chosen by Germain et al in \cite{GMScap} to study the capillary water 
waves (in this case the group velocity is $3|\xi|^{1/2}/2$). Nevertheless, in our case the term 
$\widetilde{g}(L_c)$ in $(\ref{eqcanonique})$ induces a lack of symetry which seems to limit the 
effectiveness of this approach.
\vspace{2mm}
\paragraph{Plan of the article} In section $2$ we introduce the notations and state 
our main results. Section $3$ is devoted to the reformulation of $(\ref{EK})$ as a non degenerate 
Schr\"odinger equation, and we derive the energy estimates in ``high''Sobolev spaces. We use a 
modified energy compared with \cite{Benzoni1} in order to avoid some time growth of the norms. 
In section $4$ we prove our main result in dimension at least $5$. Section $5$ begins the analysis of
dimensions $3$ and $4$, which is the heart of the paper. We only detail the case $d=3$ since $d=4$ follows 
the same ideas with simpler computations. We first introduce the functional settings, a normal form and check that it 
defines an invertible change of variable in these settings, then we bound the high order terms (at least cubic).
In section $6$ we use the method of space-time resonances (similarly to \cite{GNT3}) to bound quadratic terms and close 
the proof of global well-posedness in dimension $3$. The appendix provides some technical 
multipliers estimtes required for section $6$.

\section{Main results, tools and notations}
\label{results}
\paragraph{The results}As pointed out in the introduction, we need a condition on 
the pressure. 
\vspace{2mm}
\begin{assumption}\label{stabassump}
Throughout all the paper, we work near a constant state 
$\rho=\rho_c>0,\ u=0$, with $g'(\rho_c)>0$. 
\end{assumption}
\noindent In the case of the Euler equation, this standard condition implies 
that the linearized system 
\begin{equation*}
\left\{
\begin{array}{ll}
\partial_t\rho+\rho_c\text{div} u=0,\\
\partial_tu+g'(\rho_c)\nabla \rho=0.
\end{array}\right.
\end{equation*}
is hyperbolic, with eigenvalues (sound speed) $\pm\sqrt{\rho_c g'(\rho_c)}$.
\begin{theo}\label{theolarged}
Let $d\geq 5$, $\rho_c\in \R^{+*}$, $u_0=\n \phi_0$ be irrotational. For 
$(n,k)\in \mathbb{N},\ k>2+d/4,\ 2n+1\geq k+2+d/2$, there exists $\delta>0,$ such that if
\begin{equation*}
\|u_0\|_{H^{2n}\cap W^{k-1,4/3}}+\|\rho_0-\rho_c\|_{H^{2n+1}\cap W^{k,4/3}}\leq \delta
\end{equation*}
 then the unique solution of $(\ref{EK})$ is global with 
$\|\rho-\rho_c\|_{L^\infty(\R^+\times \R^d)}\leq \frac{\rho_c}{2}$.
\end{theo}

\begin{theo}\label{theod4}
Let $d=3$ or $4$, $u=\n \phi_0$ irrotational, $k > 2+d/4$, there exists 
$\delta>0$, $\varepsilon>0$, small enough, $n\in \mathbb{N}$ large enough, such that for
$\displaystyle \frac{1}{p}=\frac{1}{2}-\frac{1}{d}-\varepsilon$, if 
\begin{equation*}
\|u_0\|_{H^{2n}}+\|\rho_0-\rho_c\|_{H^{2n+1}}+\|xu_0\|_{L^2}+\|x(\rho_{0}-\rho_c)\|_{L^2}
+\|u_0\|_{W^{k-1,p'}}+\|\rho_{0}-\rho_c\|_{W^{k,p'}}\leq \delta,
\end{equation*}
then the solution of $(\ref{EK})$ is global with 
$\|\rho-\rho_c\|_{L^\infty(\R^+\times \R^d)}\leq\frac{\rho_c}{2}$.
\end{theo}

\begin{rmq}
While the proof implies to work with the velocity potential, we only need assumptions on the 
physical variables velocity and density.
\end{rmq}
\begin{rmq}
Actually we prove a stronger result: in the appropriate variables the solution scatters. 
Let $\mathcal{L}$ be the primitive of $\sqrt{K/\rho}$ such that $\mathcal{L}(\rho_c)=1$,
$L=\mathcal{L}(\rho)$,
$\mathcal{H}=\sqrt{-\Delta(\widetilde{g}'(1)-\Delta)}$, $\mathcal{U}
=\sqrt{-\Delta/(\widetilde{g}'(1)-\Delta)}$ ,
$\displaystyle f=e^{-it\mathcal{H}}(\mathcal{U}\phi+iL)$, then there exists 
$f_\infty$ such that 
$$\forall\,s<2n+1,\ \|f(t)-f_\infty\|_{H^{s}\cap L^2/\la x\ra}
\rightarrow_{t\rightarrow \infty} 0.$$
The analogous result is true in dimension $\geq 5$ with $t^{-d/2+1}$ for the convergence rate 
in $L^2$. See section \ref{secscatt} for a discussion in dimension $3$. It is also possible 
to quantify how large $n$ should be (at least of order  $20$, see remark \ref{precN}).
In both theorems, the size of $k$ and $n$ can be slightly 
decreased by working in fractional Sobolev spaces, but since it would remain quite large we 
chose to avoid these technicalities.
\end{rmq}

\paragraph{Some tools and notations} Most of our tools are standard analysis, except a singular 
multiplier estimate. 

\subparagraph{Functional spaces}
The usual Lebesgue spaces are $L^p$ with norm $\|\cdot\|_p$, the Lorentz spaces are $L^{p,q}$.
If $\R^+$ corresponds to the time variable, and for $B$ a Banach space, we write for short 
$L^p(\R^+,B)=L^p_tB$, similarly $L^p([0,T],B)=L^p_TB$.\\
The Sobolev spaces 
are $W^{k,p}=\{u\in L^p:\ \forall\,|\alpha|\leq k,\ D^\alpha u\in L^p\}$. 
We also use homogeneous spaces $\dot{W}^{k,p}=\{u\in L^1_{loc}:\ \forall\,|\alpha|=k,\ 
D^\alpha u\in L^p\}$. We recall the Sobolev embedding 
\begin{equation*}
\forall\,kp<d,\ \dot{W}^{k,p}(\R^d)\hookrightarrow L^{q,p}\hookrightarrow L^q,\ q=\frac{dp}{d-kq},
\ \forall\,kp>d,\ W^{k,p}(\R^d)\hookrightarrow L^\infty.
\end{equation*}
If $p=2$, as usual $W^{k,2}=H^k$, for which we have equivalent norm $\int_{\R^d}(1+|\xi|^2)^k|
\widehat{u}|^2d\xi$, we define in the usual way $H^s$ for $s\in \R$ and $\dot{H}^s$ for which 
the embeddings remain true. The following dual estimate will be of particular use
\begin{equation*}
\forall\,d\geq 3,\ \|u\|_{\dot{H}^{-1}}\lesssim \|u\|_{L^{2d/(d+2)}}.
\end{equation*}
We will use the following Gagliardo-Nirenberg type inequality (see for example \cite{TaylorIII})
\begin{equation}\label{GN1}
\forall\, l\leq p\leq k-1\text{ integers},\ 
\|D^lu\|_{L^{2k/p}}\lesssim \|u\|_{L^{2k/(p-l)}}^{(k-p)/(k+l-p)}
\|D^{k+l-p}u\|_{L^2}^{l/(k+l-p)}.
\end{equation}
and its consequence 
\begin{equation}\label{GN2}
\forall\,|\alpha|+|\beta|=k,\ \|D^\alpha fD^\beta g\|_{L^2}\lesssim \|f\|_\infty\|g\|_{\dot{H}^k}
+\|f\|_{\dot{H}^k}\|g\|_\infty.
\end{equation}
Finally, we have the basic composition estimate (see \cite{BCD}): for $F$ smooth, 
$F(0)=0$, $u\in L^\infty\cap W^{k,p}$ then\footnote{$k\in \R^+$ is allowed, but not needed.} 
\begin{equation}\label{compo}
\|F(v)\|_{W^{k,p}}\lesssim C(k,\|u\|_\infty))\|u\|_{W^{k,p}}.
\end{equation}

\subparagraph{Non standard notations}
Since we will often estimate indistinctly $z$ or $\overline{z}$, we follow the notations 
introduced in \cite{GNT3}: $z^+=z,\ z^-=\overline{z}$, and $z^\pm$ is a placeholder for $z$ or 
$\overline{z}$. The Fourier transform of $z$ is as usual $\widehat{z}$, however we also need to 
consider the profile $e^{-itH}z$, whose Fourier transform will be denoted 
$\widetilde{z^\pm}:=e^{\mp itH}\widehat{z^\pm}$.\\
When there is no ambiguity, we write $W^{k,\frac{1}{p}}$ (or $L^{\frac{1}{p}}$) instead of $W^{k,p}$ (or 
$L^{p}$) since it is convenient to use H\"older's inequality.

\subparagraph{Multiplier theorems} 
We remind that the Riesz multiplier $\nabla/|\nabla|$ is bounded on $L^p$, $1<p<\infty$.
A bilinear Fourier multiplier is defined by its symbol $B(\eta,\xi)$, it acts on $(f,g)\in 
\mathcal{S}(\R^d)$
\begin{equation*}
\widehat{B[f,g]}(\xi)=\int_{\R^d}B(\eta,\xi-\eta)\widehat{f}(\eta)\widehat{g}(\xi-\eta)d\eta.
\end{equation*}

\begin{theo}[Coifman-Meyer]
If $\partial_\xi^\alpha\partial_\eta^\beta B(\xi,\eta)\lesssim (|\xi|+|\eta|)
^{-|\alpha|-|\beta|}$, 
for sufficiently many $\alpha, \beta$ then for any $1<p,q\leq \infty$, $1/r=1/p+1/q$,
\begin{equation*}
\|B(f,g)\|_r\lesssim \|f\|_p\|g\|_q.
\end{equation*}
If moreover $\text{supp}(B(\eta,\xi-\eta))\subset \{|\eta|\gtrsim |\xi-\eta|\}$, $(p,q,r)$
are finite and $k\in \mathbb{N}$ then 
\begin{equation*}
\|\nabla^kB(f,g)\|_r\lesssim \|\nabla^kf\|_p\|g\|_q.
\end{equation*}
\end{theo}
\noindent
Mixing this result with the Sobolev embedding, we get for $2<p\leq \infty,\ 
\frac{1}{p}+\frac{1}{q}=\frac{1}{2}$ 
\begin{equation}\label{GN3}
\|fg\|_{H^s}\lesssim \|f\|_{L^p}\|g\|_{H^{s,q}}+\|g\|_{L^p}\|f\|_{H^{s,q}}\lesssim 
\|f\|_{L^p}\|g\|_{H^{s+d/p}}+\|g\|_{L^p}\|f\|_{H^{s+d/p}}.
\end{equation}

Due to the limited regularity of our multipliers,  we will need a multiplier theorem with loss 
from \cite{GuoPaus} (and inspired by corollary $10.3$ from \cite{GNT3}). Let us first describe the 
norm on symbols: for $\chi_j$ a smooth dyadic partition of the space, 
$\text{supp}(\chi_j)\subset \{2^{j-2}\leq |x|\leq 2^{j+2}\}$
\begin{equation*}
\|B(\eta,\xi-\eta)\|_{\tilde{L}^\infty_\xi\dot{B}^s_{2,1,\eta}}
=\|2^{js}\chi_j(\nabla)_\eta B(\eta,\xi-\eta)\|_{l^1(\Z,L^\infty_\xi L^2_\eta)}
\end{equation*}
The norm $\|B(\xi-\zeta,\zeta)\|_{\tilde{L}^\infty_\xi\dot{B}^s_{2,1,\zeta}}$ is defined similarly.
In practice, we rather estimate $\|B\|_{L^\infty_\xi\dot{H}^s}$ and use the interpolation estimate 
(see \cite{GNT3})
\begin{equation*}
 \|B\|_{\tilde{L}^\infty_\xi\dot{B}^s_{2,1,\eta}}\lesssim \|B\|_{L^\infty_\xi\dot{H}^{s_1}}^\theta
 \|B\|_{L^\infty_\xi\dot{H}^{s_2}}^{1-\theta},\ s=\theta s_1+(1-\theta)s_2.
\end{equation*}
We set $\|B\|_{[B^s]}=\min\big(\|B(\eta,\xi-\eta)\|_{\tilde{L}^\infty_\xi\dot{B}^s_{2,1,\eta}},
\ \|B(\xi-\zeta,\zeta)\|_{\tilde{L}^\infty_\xi\dot{B}^s_{2,1,\zeta}}\big)$.
The rough multiplier theorem is the following:
\begin{theo}[\cite{GuoPaus}]\label{singmult}
Let $0\leq s\leq d/2$, $q_1,q_2$ such that $\displaystyle \frac{1}{q_2}+\frac{1}{2}=
\frac{1}{q_1}+\bigg(\frac{1}{2}-\frac{s}{d}\bigg)$ \footnote{We write the relation between 
$(q_1,q_2)$ in a rather odd way in order to emphasize the similarity with the standard 
H\"older's inequality.}, and $\displaystyle 2\leq q_1',q_2\leq \frac{2d}{d-2s}$, then 
\begin{equation*}
\|B(f,g)\|_{L^{q_1}}\lesssim \|B\|_{[B^s]}\|f\|_{L^{q_2}}\|g\|_{L^2}.\\
\end{equation*}
Furthermore for $\displaystyle \frac{1}{q_2}+\frac{1}{q_3}=
\frac{1}{q_1}+\bigg(\frac{1}{2}-\frac{s}{d}\bigg)$, $2\leq q_i\leq \frac{2d}{d-2s}$ with $i=2,3$, 
\begin{equation*}
\|B(f,g)\|_{L^{q_1}}\lesssim \|B\|_{[B^s]}\|f\|_{L^{q_2}}\|g\|_{L^{q_3}},\\
\end{equation*}
\end{theo}

\subparagraph{Dispersion for the group $e^{-itH}$}
According to $(\ref{eqcanonique})$, the linear part of the equation reads 
$\partial_tz-i\mathcal{H}z=0$, with $\mathcal{H}=\sqrt{-\Delta(\widetilde{g}'(L_c)-\Delta)}$ 
(see also section $4$). We will use a change of variable to reduce it to 
$\widetilde{g}'(L_c)=2$, set $H=\sqrt{-\Delta(2-\Delta)}$, and use the dispersive estimate 
from \cite{GNT1}, the version in Lorentz spaces follows from real interpolation as pointed out 
in \cite{GNT3}.
\begin{theo}[\cite{GNT1}\cite{GNT3}]\label{dispersion}
For $2\leq p\leq \infty$, $s\in \R$, $U=\sqrt{-\Delta/(2-\Delta)}$, we have 
\begin{equation*}
\|e^{itH}\varphi\|_{\dot{B}^s_{p,2}}\lesssim \frac{\|U^{(d-2)(1/2-1/p)}\varphi\|
_{\dot{B}^s_{p',2}}}{t^{d(1/2-1/p)}},
\end{equation*}
and for $2\leq p<\infty$ 
\begin{equation*}
\|e^{itH}\varphi\|_{L^{p,2}}\lesssim \frac{\|U^{(d-2)(1/2-1/p)}\varphi\|
_{L^{p',2}}}{t^{d(1/2-1/p)}} 
\end{equation*}
\end{theo}
\begin{rmq}
The slight low frequency gain $U^{(d-2)(1/2-1/p)}$ is due to the fact that 
$H(\xi)=|\xi|\sqrt{2+|\xi|^2}$ behaves like $|\xi|$ at low frequencies, which has a strong
angular curvature and no radial curvature.
\end{rmq}
\begin{rmq}
Combining the dispersion estimate and the celebrated $TT^*$ argument, Strichartz estimates follow
$$
\|e^{itH}\varphi\|_{L^pL^q}\lesssim \|U^{\frac{d-2}{2}(1/2-1/p)}\varphi\|_{L^2},\ \frac{2}{p}
+\frac{d}{q}=\frac{d}{2},\ 2\leq p\leq \infty,
$$
however the dispersion estimates are sufficient for our purpose.
\end{rmq}

\section{Reformulation of the equations and energy estimate}\label{secenergie}
As observed in \cite{Benzoni1}, setting $w=\sqrt{K/\rho}\n \rho$, $\mathcal{L}$ the primitive of 
$\sqrt{K/\rho}$ such that $\mathcal{L}(\rho_c)=1$, 
$L=\mathcal{L}(\rho)$, $z=u+iw$ the Euler-Korteweg system rewrites 
\begin{eqnarray*}
\partial_tL+u\cdot\nabla L+a(L)\text{div}u&=&0,\\
\partial_tu+u\cdot \nabla u-w\cdot \nabla w -\nabla (a(L)\text{div}w)&=&-\widetilde{g}'(L)w,\\
\partial_tw+\nabla(u\cdot w)+\nabla (a(L)\text{div}u)&=&0,\\
\end{eqnarray*}
where the third equation is just the gradient of the first.
Setting $l=L-1$, in the potential case $u=\nabla\phi$, the system on $\phi,l$ then reads
\begin{equation}\label{EKpot}
 \left\{
\begin{array}{lll}
\displaystyle \partial_t\phi+\frac{1}{2}\big(|\nabla \phi|^2-|\nabla l|^2\big)-a(1+l)\Delta l=
-\widetilde{g}(1+l),\\
\displaystyle \partial_tl+\nabla \phi\cdot \nabla l+a(1+l)\Delta \phi=0,
\end{array}\right.
\end{equation}
with $\widetilde{g}(1)=0$ since we look for integrable functions. As a consequence of the stability 
condition $(\ref{stabassump})$, up to a change of variables we can and will assume through the 
rest of the paper that 
\begin{equation}\label{valeurg'}
\widetilde{g}'(1)=2. 
\end{equation}
The number $2$ has no significance except that this choice gives the same linear part as for 
the Gross-Pitaevskii equation linearized near the constant state $1$.
\begin{prop}\label{energy}
Under the following assumptions 
\begin{itemize}
 \item $(\n \phi_0,l)\in H^{2n}\times H^{2n+1}$
 \item Normalized $(\ref{stabassump})$: $\widetilde{g}'(1)=2$ 
 \item $L(x,t)=1+l(x,t)\geq m>0$ for $(x,t)\in \R^d\times [0,T]$,
 \end{itemize}
then for $n>d/4+1/2$, there exists a continuous function $C$ such that the solution of $(\ref{EKpot})$ satisfies 
the following estimate
\begin{equation*}
\begin{aligned}
&\|\nabla \phi\|_{H^{2n}}+\|l\|_{H^{2n+1}}\\
&\leq
 \big(\|\nabla \phi_0\|_{H^{2n}}+\|l_0\|_{H^{2n+1}}\big)
 \rm{exp}\bigg(\int_0^t C(\|l\|_{L^\infty},\|\frac{1}{l+1}\|_{L^\infty},\|z\|_{L^\infty})\\
 &\hspace{75mm}\times(\|\nabla\phi(s)\|_{W^{1,\infty}} +\|l(s)\|_{W^{2,\infty}})ds\bigg),
 \end{aligned}
\end{equation*}
where $z(s)=\nabla\phi(s)+i\nabla w(s)$.
\end{prop}
This is almost the same estimate as in \cite{Benzoni1} but for an essential point: in the integrand
of the right hand side there is no constant added to $\|\nabla \phi(s)\|_{W^{1,\infty}}
+\|l(s)\|_{W^{2,\infty}}$, the price to pay is that we can not control $\phi$ but its gradient (this is naturel since the difficulty is related to the low frequencies).
Before going into the detail of the computations, let us underline on a very simple example 
the idea behind it. We consider the linearized system 
\begin{eqnarray}
\label{lin1}
\p_t\phi-\Delta l+2l=0,\\
\label{lin2}
\p_tl+\Delta \phi=0. 
\end{eqnarray}
Multiplying $(\ref{lin1})$ by $\phi$, $(\ref{lin2})$ by $l$, integrating and using Young's 
inequality leads to the ``bad'' estimate
\begin{equation*}
 \frac{d}{dt}\big(\|\phi\|_{L^2}^2+\|l\|_{L^2}^2\big)\lesssim 2(\|\phi\|_{L^2}^2+\|l\|_{L^2}^2),
\end{equation*}
on the other hand if we multiply $(\ref{lin1})$ by $-\Delta\phi$, 
$(\ref{lin2})$ by $(-\Delta+2)l$ we get
\begin{equation*}
\frac{d}{dt}\int_{\R^d}(\frac{|\n l|^2+|\n\phi|^2}{2}+l^2)dx=0,
\end{equation*}
the proof that follows simply mixes this observation with the gauge method from \cite{Benzoni1}.
\begin{proof}
Let us start with the equation on $z=\nabla \phi+i\nabla l=u+iw$, we remind that 
$\widetilde{g}'(1)=2$, so that we write it
\begin{equation}\label{schroddeg}
\partial_tz+z\cdot \nabla z+i\nabla (a \text{div}z)=-2w+(2-\widetilde{g}'(1+l))w.
\end{equation}
We shortly recall the method from \cite{Benzoni1} that we will slightly simplify since we do not 
need to work in fractional Sobolev spaces. Due to the quasi-linear nature of the system
(and in particular the bad ``non transport term'' $iw\cdot \n z$), it is not possible to directly estimate $\|z\|_{H^{2n}}$ by energy estimates, instead 
one uses a gauge function $\varphi_n(\rho)$ and control $\|\varphi_n\Delta^nz\|_{L^2}$. When we take 
the product of $(\ref{schroddeg})$ with $\varphi_n$ real, a number of 
commutators appear:
\begin{equation}\label{C1}
\varphi_n\Delta^n\partial_tz=\partial_t(\varphi_n\Delta^nz)-(\partial_t\varphi_n)\Delta^nz
=\partial_t(\varphi_n\Delta^nz)+C_1
\end{equation}
\begin{equation}\label{C2}
\varphi_n\Delta^n(u\cdot \nabla z)=u\cdot \nabla (\varphi_n\Delta^nz)+[\varphi_n\Delta^n,u\cdot\n]
z:=u\cdot \n(\varphi_n\Delta^n z)+C_2
\end{equation}
\begin{equation}\label{C3}
i\varphi_n\Delta^n(w\cdot \nabla z)=iw\cdot \nabla (\varphi_n\Delta^nz)+
[\varphi_n\Delta^n,w\cdot\n]z:=iw\cdot \nabla (\varphi_n\Delta^nz)+C_3,
\end{equation}
The term $\n (a\text{div}z)$ requires a bit more computations:
\begin{eqnarray*}
 i\varphi_n\Delta^n\n (a\text{div}z)=i\n (\varphi_n\Delta^n (a \text{div}z))-i(\n \varphi_n)
 \Delta^n(a\text{div}z),
\end{eqnarray*}
then using recursively $\Delta (fg)=2\n f\cdot\n g+f\Delta g+(\Delta f)g$ we get 
\begin{equation*}
\Delta^n (a\text{div}z)=a\text{div}\Delta^n z+2n(\n a)\cdot \Delta^{n}z+C,
\end{equation*}
where $C$ contains derivatives of $z$ of order at most $2n-1$, so that 
\begin{eqnarray}
\nonumber
i\varphi_n\Delta^n\n (a\text{div}z)&=&i\n \bigg(\varphi_n\big(a\text{div}\Delta^n z
 +2n(\n a)\cdot \Delta^{n}z\big)\bigg)-i\n \varphi_n a\text{div} \Delta^n z+i\n (\varphi_n C)\\
\nonumber  &=&i\nabla \big(a\text{div}(\varphi_n\Delta^nz)\big)
+2in\nabla a\cdot \varphi_n\nabla \Delta^nz-ia(\nabla+I_d\text{div}) \Delta^nz\cdot \nabla 
\varphi_n\\
\label{C4}&&+C_4,
\end{eqnarray}
where $C_4$ contains derivatives of $z$ of order at most $2n$ and by notation $I_d\text{div}\Delta^nz\cdot \nabla 
\varphi_n={\rm div}\D^n z\,\n\va	_n$. Finally, we define 
$C_5=-\varphi_n\Delta^n\big((2-\widetilde{g}'(1+l))w\big)$. The equation on 
$\varphi_n\Delta^nz$ thus reads 
\begin{eqnarray}
\p_t(\varphi_n\Delta^nz)+u\cdot \n (\varphi_n\Delta^nz)+i\n \big(a\text{div}(\varphi_n\Delta^nz)
\big)+iw u\cdot \n (\varphi_n\Delta^nz)+ 2\varphi_n\Delta^nw=\\
-\sum_1^5C_k-2in\varphi_n\nabla \Delta^nz\cdot \nabla a+ia(\nabla+I_d\text{div}) \Delta^nz\cdot \nabla 
\varphi_n
\end{eqnarray}
Taking the scalar product with $\varphi_n\Delta^nz$, integrating and taking the real part gives 
for the first three terms
\begin{equation}
\frac{1}{2}\frac{d}{dt}\int_{\R^d}(\varphi_n\Delta^nz)^2dx-\frac{1}{2}\int_{\R^d} \text{div}u
|\varphi_n\Delta^nz|^2dx.
\end{equation}
And we are left to control the remainder terms from $(\ref{C3},\, \ref{C4})$. Using 
$w=\frac{a}{\rho} \nabla \rho$, $\varphi_n=\varphi_n(\rho)$, we rewrite 
\begin{eqnarray*}
i\varphi_nw\cdot \nabla (\Delta^nz)+2ni\varphi_n\nabla (\Delta^nz)\cdot \nabla a-ia \nabla
(\Delta^nz)\cdot \nabla \varphi_n-ia\nabla \varphi_n \,\text{div}\Delta^nz
\\
=i\varphi_n\bigg(w\cdot \nabla -\frac{a\nabla \varphi_n}{\varphi_n}\cdot \nabla 
-\frac{a\nabla \varphi_n}{\varphi_n}\text{div}+2n\nabla a\cdot\n\bigg)\Delta^nz.
\end{eqnarray*}
\begin{equation}\label{reste}
=i\varphi_n\bigg[\bigg(\frac{a}{\rho}-a\frac{\varphi_n'}{\varphi_n}\bigg)\nabla \rho\cdot\nabla-
\frac{a\varphi_n'}{\varphi_n}\nabla \rho\,\text{div}+2na'\nabla \rho\cdot \nabla\bigg]
\Delta^nz
\end{equation}
If the $\text{div}$ operator was a gradient, the most natural choice for $\varphi_n$ would be
to take
\begin{equation*}
 \frac{a}{\rho}-\frac{2a\varphi_n'}{\varphi_n}+2na'=0\Leftrightarrow \frac{\varphi_n'}{\varphi_n}
 =\frac{1}{2\rho}+\frac{na'}{a}\Leftarrow \varphi_n(\rho)=a^n(\rho)\sqrt{\rho}.
\end{equation*}
For this choice the remainder $(\ref{reste})$ rewrites
\begin{equation*}
\bigg[\bigg(\frac{a}{\rho}-a\frac{\varphi_n'}{\varphi_n}\bigg)\nabla \rho\cdot\nabla-
\frac{a\varphi_n'}{\varphi_n}\nabla \rho\text{div}+2na'\nabla \rho\cdot \nabla\bigg]\Delta^nz
=\bigg(\frac{a}{2\rho}+na'\bigg)\nabla \rho\cdot (\n-I_d\text{div})\Delta^nz.
\end{equation*}
Using the fact that $\varphi_n(a/(2\rho)+na')(\rho)\n \rho$ is a real valued gradient, and setting 
$z_n=\Delta^nz$, we see that the contribution of $(\ref{reste})$ in the energy estimate is 
actually $0$ from the following identity (with the Hessian $\text{Hess}H$):
\begin{eqnarray*}
\text{Im}\int_{\R^d} \overline{z_n}\cdot (\n-I_d\text{div})z_n\cdot \nabla H(\rho)dx
&=&\text{Im}\int_{\R^d}\overline{z_{i,n}}\p_jz_{i,n}\p_j H-\overline{z_{i,n}}\p_jz_{j,n}\p_iH
\\
&=&
\text{Im}\int_{\R^d} \overline{z_n}\text{Hess}Hz_n-\D H|z_n|^2 \\
&&\hspace{2cm}-\partial_jH z_{i,n}(\overline{\partial_jz_{i,n}}
-\overline{\partial_iz_{j,n}})dx\\
&=&0.
\end{eqnarray*}
We have used the fact that $z$ is irrotationnal.
Finally, we have obtained 
\begin{equation}\label{avantdernier}
\frac{1}{2}\frac{d}{dt}\int \|\varphi_n\Delta^nz\|_{L^2}^2dx-\frac{1}{2}\int_{\R^d}(\text{div}u)
|\varphi_n\Delta^nz|^2
=-\int\sum_1^5C_k\varphi_n\Delta^n\overline{z}dx-2\int\varphi_n^2 \Delta^nw\Delta^nu\,dx.
\end{equation}
Note that the terms $C_k\varphi_n\Delta^nz$ are cubic while $\varphi_n\Delta^nw\Delta^nu$ is 
only quadratic, thus we will simply bound the first ones while we will need to cancel the later.
\paragraph{Control of the $C_k$ :}
From their definition, it is easily seen that the $(C_k)_{2\leq i\leq 4}$ only contain terms of 
the kind $\partial^\alpha f\partial^\beta g$ with $f,g=u$ or $w$, $|\alpha|+|\beta|\leq 2n$, thus 
\begin{equation*}
\forall\, 2\leq k\leq 4,\ \bigg|\int C_k\varphi_n\Delta^nzdx\bigg|\lesssim 
\sum_{|\alpha|+|\beta|=2n,\ f,g=u\text{ or }w}
\|\p^\alpha f\p^\beta g\|_{L^2}\|z\|_{H^{2n}}
\end{equation*}
When $|\alpha|=0,\ |\beta|=2n$, we have obviously 
$\|f\partial^\beta g\|_{L^2}\lesssim \|f\|_\infty\|g\|_{H^{2n}}$, while the general case
$\|\partial^\alpha f\partial^\beta g\|_2\lesssim \|f\|_\infty\|g\|_{H^{2n}}
+\|g\|_\infty\|f\|_{H^{2n}}$
is Gagliardo-Nirenberg'interpolation inequality $(\ref{GN2})$. We deduce 
\begin{equation*}
\forall\, 2\leq k\leq 4,\ \bigg|\int C_k\varphi_n\Delta^nzdx\bigg|\lesssim 
\|z\|_\infty\|z\|_{H^{2n}}^2.
\end{equation*}
Let us deal now with $C_1=-\p_t\va_n\D^n z$, since $\partial_t\varphi_n=-\varphi_n'\text{div}(\rho u)$ we have
\begin{equation*}
 \bigg|\int_{\R^d}C_1\varphi_n\Delta^n\overline{z}dx\bigg|\lesssim F((\|l\|_{L^\infty},\|\frac{1}{l+1}\|_{L^\infty}) (\|u\|_{W^{1,\infty}}+\|z\|_{L^\infty}^2)
 \|z\|_{H^{2n}}^2
\end{equation*}
with $F$ a continuous function.\\
We now estimate the contribution of $C_5=-\varphi_n\Delta^n\big((2-\widetilde{g}'(1+l))w\big)$: since $\widetilde{g}'(1)=2$, from the composition rule
$(\ref{compo})$ we have $\|\widetilde{g}'(1+l)-2\|_{H^{2n}}\lesssim F_1(\|l\|_{L^\infty},\|\frac{1}{l+1}\|_{L^\infty}) \|l\|_{H^{2n}}$ with $F_1$ a continuous function with $F_1(0,\cdot)=0$
so that
\begin{equation*}
\begin{aligned}
&\bigg|\int_{\R^d}C_5\varphi_n\Delta^n \overline{z} dx\bigg|\lesssim  \|(2-\widetilde{g}')w\|_{H^{2n}}
\|z\|_{H^{2n}}
\lesssim (\|(2-\widetilde{g}'(1+l))\|_{L^\infty}\|z\|_{H^{2n}}\\
&\hspace{5cm}+ F_1(\|l\|_{L^\infty},\|\frac{1}{l+1}\|_{L^\infty})\|l\|_{H^{2n}}\|z\|_\infty)\|z\|_{H^{2n}}.
\end{aligned}
\end{equation*}
To summarize, for any $1\leq k\leq 5$, we have 
\begin{equation}\label{estimCk}
\bigg|\int_{\R^d}C_k\varphi_n\Delta^n z dx\bigg| 
\lesssim F_2(\|l\|_{L^\infty},\|\frac{1}{l+1}\|_{L^\infty})(\|l\|_\infty+\|z\|_{W^{1,\infty}}+\|z\|_{L^\infty}^2)(\|l\|_{H^{2n}}^2+\|z\|_{H^{2n}}^2),
\end{equation}
 with $F_2$ a continuous function.

\paragraph{Cancellation of the quadratic term} 
We start with the equation on $l$ to which we apply $\varphi_n\Delta^n$, multiply by 
$\varphi_n(\Delta^nl)/a$ and
integrate in space
\begin{eqnarray*}
\int_{\R^d}\frac{\varphi_n^2}{a}\Delta^nl\partial_t\Delta^nl+\frac{\va_n^2}{a}(\Delta^nl)
\Delta^n(\nabla \phi\cdot \nabla l)+\varphi_n^2\Delta^nl\frac{\Delta^n(a\Delta \phi)}{a}=0.
\end{eqnarray*}
Commuting $\Delta^n$ and $a$, and using an integration by part, this rewrites
\begin{equation*}
\begin{aligned}
&\displaystyle \frac{1}{2}\frac{d}{dt}\int_{\R^d}\frac{\varphi_n^2}{a}(\Delta^nl)^2dx-\int_{\R^d}\frac{d}{dt}(\frac{\varphi_n^2}{2a})|\Delta^nl|^2 dx+\int_{\R^d}\frac{\va_n^2}{a}(\Delta^nl)
\Delta^n(\nabla \phi\cdot \nabla l)\\
&\hspace{6cm}+\int_{\R^d}\varphi_n^2\Delta^n l\,\D\Delta^n \phi dx+\frac{\varphi_n^2}{a}\Delta^nl[\Delta^n,a]
\Delta \phi dx\\[2mm]
&\displaystyle \frac{1}{2}\frac{d}{dt}\int_{\R^d}\frac{\varphi_n^2}{a}(\Delta^nl)^2dx-\int_{\R^d}\frac{d}{dt}(\frac{\varphi_n^2}{2a})|\Delta^nl|^2 dx+\int_{\R^d}\frac{\va_n^2}{a}(\Delta^nl)
\Delta^n(\nabla \phi\cdot \nabla l)\\
&-\int_{\R^d}\varphi_n^2\,\n\Delta^n l\cdot\n\Delta^n \phi \,dx- \int_{\R^d}\Delta^n l\, \n \varphi_n^2\cdot\n\Delta^n \phi \,dx+\frac{\varphi_n^2}{a}\Delta^nl[\Delta^n,a]
\Delta \phi dx\\
\end{aligned}
\end{equation*}
We remark that the integrand in the right hand side only depends on $l,\n\phi$ and their 
derivatives, therefore using the same commutator arguments as previously, we get the bound 
\begin{equation}
\label{suprquad}
\begin{aligned}
&\frac{1}{2}\frac{d}{dt}\int_{\R^d}\frac{\varphi_n^2}{a}(\Delta^nl)^2dx
-\int_{\R^d}\varphi_n^2(\Delta^n\n \phi)\Delta^n\n ldx\\
&\hspace{2cm}\lesssim F_3(\|l\|_{L^\infty},\|\frac{1}{l+1}\|_{L^\infty})(\|l\|_\infty+\|z\|_{W^{1,\infty}}+\|z\|_{L^\infty}^2)(\|l\|_{H^{2n}}^2+\|z\|_{H^{2n}}^2),
\end{aligned}
\end{equation}
 with $F_3$ a continuous function.
 Now if we add $(\ref{avantdernier})$ to $2\times(\ref{suprquad})$ and use the estimates on $(C_k)$
we obtain 
\begin{equation*}
\begin{aligned}
&\frac{1}{2}\frac{d}{dt}\int \|\varphi_n\Delta^nz\|_{L^2}^2+\|\Delta^nl\|_{L^2}^2dx\\
&\hspace{2cm}\lesssim F_4(\|l\|_{L^\infty},\|\frac{1}{l+1}\|_{L^\infty})(\|l\|_\infty+\|z\|_{W^{1,\infty}}+\|z\|_{L^\infty}^2)(\|l\|_{H^{2n}}^2+\|z\|_{H^{2n}}^2),
\end{aligned}
\end{equation*}
 with $F_4$ a continuous function. The conclusion then follows from Gronwall's lemma.
\end{proof}
\section{Global well-posedness in dimension larger than 4}\label{sectiondlarge}
We first make a further reduction of the equations that will be also used for the cases $d=3,4$,
namely we rewrite it as a linear Schr\"odinger equation with some remainder. In addition to 
$\widetilde{g}'(1)=2$, we can also assume $a(1)=1$, so that $(\ref{EKpot})$ 
rewrites\footnote{The assumption $a(1)=1$ should add some constants in factor of the nonlinear 
terms, we will neglect it as it will be clear in the proof that multiplicative constants do not 
matter.}
\begin{equation}
 \left\{
\begin{array}{lll}
\displaystyle \partial_t\phi-\Delta l+2l=(a(1+l)-1)\Delta l
-\frac{1}{2}\big(|\nabla \phi|^2-|\nabla l|^2\big)
+(2l-\widetilde{g}(1+l)),\\
\displaystyle \partial_tl+\Delta \phi=-\nabla \phi\cdot \nabla l+(1-a(1+l))\Delta \phi.
\end{array}\right.
\label{systedepart}
\end{equation}
The linear part precisely corresponds to the linear part of the Gross-Pitaevskii equation. 
In order to diagonalize it, following \cite{GNT1} we set 
\begin{equation*}
U=\sqrt{\frac{-\Delta}{2-\Delta}},\ H=\sqrt{-\Delta(2-\Delta)},\ \phi_1=U\phi,\ l_1=l.
\end{equation*}
The equation writes in the new variables
\begin{equation}\label{ekdiag}
 \left\{
\begin{array}{lll}
\displaystyle \partial_t\phi_1+Hl_1=U\bigg((a(1+l_1)-1)\Delta l_1
-\frac{1}{2}\big(|\nabla U^{-1} \phi_1|^2-|\nabla l_1|^2\big)
+(2l_1-\widetilde{g}(1+l_1))\bigg),\\
\displaystyle \partial_tl_1-H \phi_1=-\nabla U^{-1}\phi_1\cdot \nabla l_1-(1-a(1+l_1))H U^{-1}\phi_1.
\end{array}\right.
\end{equation}
More precisely, if we set $\psi=\phi_1+il_1,\ \psi_0=(U\phi+il)|_{t=0}$, the Duhamel formula 
gives 
\begin{eqnarray}\label{duhamel}
\psi(t)&=&e^{itH}\psi_0+\int_0^te^{i(t-s)H}\mathcal{N}(\psi(s))ds,\\
\nonumber 
\text{with }\mathcal{N}(\psi)&=&U\big((a(1+l_1)-1)\Delta l_1-\frac{1}{2}\big(|\nabla U^{-1} 
\phi_1|^2-|\nabla l_1|^2\big)+(2l_1-\widetilde{g}(1+l_1))\big)\\
&&+i\big(-\nabla U^{-1}\phi_1\cdot \nabla l_1-\big(1-a(1+l_1)\big)H\phi
\big).
\end{eqnarray}
We underline that for low frequencies the situation is more favorable than for 
the Gross-Pitaevskii equation, as all the terms where $U^{-1}$ appears already contain derivatives 
that compensate this singular multiplier. Note however that the Gross-Pitaevskii equations are 
formally equivalent to this system via the Madelung transform in the special case $K(\rho)=\kappa
/\rho$, so our computations are a new way of seeing that these singularities can be removed in 
appropriate variables.
Let us now state the key estimate:

\begin{prop}\label{decay}
Let $d\geq 5$, $T>0$, $k\geq 2$, $N\geq k+2+d/2$, we set
$\displaystyle \|\psi\|_{X_T}=\|\psi\|_{L^\infty([0,T], H^N)}+
\sup_{t\in [0,T]}(1+t)^{d/4}\|\psi(t)\|_{W^{k,4}}$,
then the solution of $(\ref{duhamel})$ satisfies 
\begin{equation*}
\forall\,t\in [0,T],\ \|\psi(t)\|_{W^{k,4}}\lesssim \frac{\|\psi_0\|_{W^{k,4/3}}+\|\psi_0\|_{H^N}
+G(\|\psi\|_{X_t},\|\frac{1}{1+l_1}\|_{L^\infty_t(L^\infty)})\|\psi\|_{X_T}^2}{(1+t)^{d/4}},
\end{equation*}
with $G$ a continuous function.
\end{prop}

\begin{proof}
We start with $(\ref{duhamel})$. From the dispersion estimate $(\ref{dispersion})$ and the Sobolev 
embedding, we have for any $t\geq 0$
\begin{equation*}
(1+t)^{d/4}\|e^{itH}\psi_0\|_{W^{2,4}}\lesssim 
(1+t)^{d/4}\min\bigg(\frac{\|U^{(d-2)/4}\psi_0\|_{W^{2,4/3}}}{t^{d/4}},\,
\|\psi_0\|_{H^N}\bigg)
\lesssim \|\psi_0\|_{W^{2,4/3}}+\|\psi_0\|_{H^N}.
\end{equation*}
The only issue is thus to bound the nonlinear part. Let $f,g$ be a placeholder for $l_1$ or 
$U^{-1}\phi_1$, there are several kind of terms : $\n f\cdot \n g$, $(a(1+l_1)-1)\Delta f$,  $2l_1-\widetilde{g}(1+l_1)$, $|\n f|^2$, $|\n g|^2$, $(a(1+l_1)-1)H g$. The estimates for $0\leq t\leq 1$ are easy (it corresponds to the existence of strong solution in finite time), so we assume $t\geq 1$ and 
we split the integral from $(\ref{duhamel})$ between $[0,t-1]$ and $[t-1,t]$. For the first kind 
we have from the dispersion estimate and (\ref{GN3}):
\begin{eqnarray*}
\bigg\|\int_0^{t-1}e^{i(t-s)H}\n f\cdot \n g\,ds\bigg\|_{W^{k,4}}
&\lesssim& \int_0^{t-1}\frac{\|\n f\cdot \n g\|_{W^{k,4/3}}}{(t-s)^{d/4}}ds
\\
&\lesssim&
\int_0^{t-1} \frac{\|\n f\|_{H^{k}}\|\nabla g\|_{W^{k-1,4}}}{(t-s)^{d/4}}ds,\\
&\lesssim&\|\psi\|_{X_t}^2 \displaystyle \int_0^{t-1}
\frac{1}{(t-s)^{d/4}(1+s)^{d/4}}ds\\
&\lesssim& \frac{\|\psi\|_{X_t}^2}{t^{d/4}}.
\end{eqnarray*}
(actually we should also add on the numerator $\|\n f\|_{W^{k-1,4}}\|\nabla g\|_{H^{k}}$, but 
since $f,g$ are symmetric placeholders we omit this term). 
We have used the fact that $\n U^{-1}$ is 
bounded on $W^{1,p}\rightarrow L^p$, 
$1<p<\infty$ so that $\|\n f(s)\|_{H^k}\lesssim \|f\|_{X_t}$ for $s\in [0,t]$, $(1+s)^{d/4}\|\n g\|_{W^{k-1,4}}
\lesssim \|g\|_{X_t}$.\\
For the second part on $[t-1,t]$ we use the Sobolev embedding $H^{d/4}\hookrightarrow L^4$ and 
(\ref{GN3}):
\begin{eqnarray*}
\bigg\|\int_{t-1}^te^{i(t-s)H}(\nabla f\cdot \n g )ds\bigg\|_{W^{k,4}}\lesssim 
\int_{t-1}^t\big\|\n f\cdot\n g\big\|_{H^{k+d/4}} ds&\lesssim& 
 \int_{t-1}^t\|\n f\|_{L^4}\|\n g\|_{H^{k+d/2}}ds\\
&\lesssim& 
\|\psi\|_{X_t}^2 \int_{t-1}^t \frac{1}{(1+s)^{\frac{d}{4}}}ds\\
&\lesssim& \frac{\|\psi\|_{X_t}^2 }{(1+t)^{d/4}}.
\end{eqnarray*}
The terms of the kind $(a(1+l_1)-1)\Delta f$ are estimated similarly: splitting the 
integral over $[0,t-1]$ and $[t-1,t]$,
\begin{eqnarray*}
\bigg\|\int_0^{t-1}e^{i(t-s)H}(a(1+l_1)-1)\Delta fds\bigg\|_{W^{k,4}}
&\lesssim&
\int_0^{t-1} \frac{\|a(1+l_1)-1\|_{W^{k,4}}\|\Delta f\|_{H^k}}{(t-s)^{d/4}}ds
\\
&\lesssim& 
\int_0^{t-1} \frac{\|a(1+l_1)-1\|_{W^{k,4}}\|\n f\|_{H^{k+1}}}{(t-s)^{d/4}}ds.
\end{eqnarray*}
As for the first kind terms, from the composition estimate we deduce that:
$$\|a(1+l_1)-1\|_{W^{k,4}}\lesssim 
F(\|l_1\|_{L^\infty_t(L^\infty)},\|\frac{1}{1+l_1}\|_{L^\infty_t(L^\infty)})\|l_1\|_{W^{k,4}},$$
 with $F$ continuous, we can bound the integral above by $F(\|\psi\|_{X_t},\|\frac{1}{1+l_1}\|_{L_t^\infty(L^\infty)})
 \|\psi\|_X^2/t^{5/4}$. For the integral 
over $[t-1,t]$ we can again do the same computations using 
the composition estimates $ 
\|a(1+l_1)-1\|_{H^{k+d/2}}\lesssim F_1(\|l_1\|_{L^\infty_t(L^\infty)},\|\frac{1}{1+l_1}\|_{L^\infty_t(L^\infty)}) \|l_1\|_{H^{k+d/2}}$ with $F_1$ continuous. The restriction $N\geq k+2+d/2$ comes 
from the fact that we need $\|\Delta f\|_{H^{k+d/2}}\lesssim \|f\|_X$.\\
Writing $2l_1-\widetilde{g}(1+l_1)=l_1(2-\widetilde{g}(l_1)/l_1)$ we see that the estimate for the 
last term is the same as for $(a(1+l_1)-1)\Delta f$ but simpler so we omit it. The other terms can be also handled in a similar way.
\end{proof}

\paragraph{End of the proof of theorem $(\ref{theolarged})$} We fix $k>2+d/4$, $n$ such that 
$2n+1\geq k+2+d/2$, and use these values for $X_T=L^\infty([0,T],H^{2n+1}\cap 
(1+t)^{-d/4}W^{k,4})$. 
First note that since $\mathcal{L}$ is a smooth diffeomorphism near $1$ and $u_0=\n \phi_0$, 
we have
\begin{eqnarray*}
\|u_0\|_{H^{2n}\cap W^{k-1,4/3}}+\|\rho_0-\rho_c\|_{H^{2n+1}\cap W^{k,4/3}}&\sim&
\|(U\phi_0,\mathcal{L}^{-1}(1+l_0)-1)\|_{(H^{2n+1}\cap W^{k,4/3})^2}
\\
&\sim& \|\psi_0\|_{H^{2n+1}\cap W^{k,4/3}},
\end{eqnarray*}
if $\|l_0\|_\infty$ is small enough. In particular we will simply write the smallness 
condition in term of $\psi_0$. Now using the embedding $W^{k,4}\hookrightarrow W^{2,\infty}$, 
the energy estimate of proposition $(\ref{energy})$ implies
\begin{equation*}
\|\psi(t)\|_{H^{2n+1}}\leq \|\psi_0\|_{H^{2n+1}}\text{exp}\bigg(C\int_0^t H(\|\psi\|_{X_s},\|\frac{1}{l+1}\|_{L^\infty})( \|\psi\|_{W^{k,4}}+\|\psi\|_{W^{k-1,4}}^2)  ds
\bigg).
\end{equation*}
Combining it with the decay estimate of proposition $(\ref{decay})$ we get with $G$ and $H$ continuous:
\begin{eqnarray*}
\begin{aligned}
&\|\psi\|_{X_T}\leq C_1\bigg(\|\psi_0\|_{W^{k,4/3}}+\|\psi_0\|_{H^{2n+1}}+\|\psi\|_{X_T}^2 G(\|\psi\|_{X_T},\|\frac{1}{1+l_1}\|_{L^\infty_T(L^\infty)})\\
&+\|\psi_0\|_{H^N}\text{exp}\bigg(C\int_0^T H(\|\psi\|_{X_T},\|\frac{1}{l+1}\|_{L^\infty_T(L^\infty)})( \|\psi\|_{W^{k,4}}+\|\psi\|_{W^{k-1,4}}^2)  ds\\
&\leq C_1\bigg(\|\psi_0\|_{W^{k,4/3}}+\|\psi_0\|_{H^N}+\|\psi\|_{X_T}^2 G(\|\psi\|_{X_T},\|\frac{1}{1+l_1}\|_{L^\infty_T(L^\infty)})\\
&\hspace{3cm}+ \|\psi_{0}\|_{H^{2n+1}}\text{exp}\big(C'\|\psi\|_{X_T} H(\|\psi\|_{X_T},\|\frac{1}{l+1}\|_{L^\infty_T(L^\infty)})\big)\bigg).
\end{aligned}
\end{eqnarray*}
From the usual bootstrap argument, we find that for  
$\|\psi_0\|_{W^{k,4/3}}+\|\psi_0\|_{H^N}\leq \varepsilon$ small enough then for any 
$T>0$, $\|\psi\|_{X_T}\leq 3C_1\varepsilon$ (it suffices to note that for $\varepsilon$ small 
enough, the application $m\mapsto C_1(\varepsilon+\varepsilon e^{C'm}+m^2)$ is smaller than $m$ on 
some interval $[a,b]\subset ]0,\infty[$ with $a\simeq 2C_1\varepsilon$). \\
In particular $\|l\|_\infty\lesssim \varepsilon$ and up to diminishing $\varepsilon$, we have 
$$\|\rho-\rho_c\|_{L^\infty([0,T]\times \R^d)}=\|\mathcal{L}^{-1}(1+l)-\rho_c\|_\infty 
\leq \rho_c/2.$$
This estimate and the $H^{2n+1}$ bound allows to apply the blow-up criterion of \cite{Benzoni1} 
to get global well-posedness.

\section[The case of dimension 3]{The case of dimension d=3,4: normal form, bounds for cubic 
and quartic terms}
In dimension $d=4$ the approach of section $4$ fails, and $d=3$ is even worse. Thus we 
need to study more carefully the structure of the nonlinearity. We start with $(\ref{ekdiag})$, 
that we rewrite in complex form 
\begin{eqnarray}
\nonumber
\partial_t\psi-iH\psi&=&U\big[(a(1+l)-1)\Delta l
-\frac{1}{2}\big(|\nabla \phi|^2-|\nabla l|^2\big)
+(2l-\widetilde{g}(1+l))\big]\\
\nonumber &&+i\big[-\nabla \phi\cdot \nabla l+\big(1-a(1+l)\big)\Delta\phi)\big]\\
\label{GPcomp}&=& U\mathcal{N}_1(\phi,l)+i\mathcal{N}_2(\phi,l)=\mathcal{N}(\psi).
\end{eqnarray}
As explained in the introduction (see \eqref{duhaprofile}), we can rewrite the Duhamel 
formula in term of the profile $e^{-itH}\psi$. In particular, (the Fourier transform of) 
quadratic terms read
\begin{equation}\label{genericquad}
I_{\text{quad}}=
e^{itH(\xi)}\int_0^te^{-is \big(H(\xi)\mp H(\eta)\mp H(\xi-\eta)\big)}B(\eta,\xi-\eta)
\widetilde{\psi^\pm}(\eta)\widetilde{\psi^\pm}(\xi-\eta)d\eta ds,
\end{equation}
where we remind the notation $\widetilde{\psi^\pm}=e^{\mp itH}\widehat{\psi^\pm}$, and $B$ is 
the symbol of a bilinear multiplier. 
For some $\varepsilon>0$ to choose later, $1/p=1/6-\varepsilon$, $T>0$ we set with $N=2n+1$:
\begin{equation}\label{Fspaces}
 \left\{
 \begin{array}{lll}
 \|\psi\|_{Y_T}&=&\|xe^{-itH}\psi\|_{L^\infty_T( L^2})
+\|\ct^{1+3\varepsilon}\psi\|_{L^\infty_T(W^{k,p})},\\
 \|\psi\|_{X(t)}&=&\|\psi(t)\|_{H^N}+\|xe^{-itH}\psi(t)\|_{L^2}
+\|\ct^{1+3\varepsilon}\psi(t)\|_{W^{k,p}},\\
 \|\psi\|_{X_T}&=&\displaystyle \sup_{[0,T)}\|\psi\|_{X(t)}.
\end{array}\right.
\end{equation}
From the embedding $W^{3,p}\subset W^{2,\infty}$, proposition \ref{energy} implies
\begin{equation*}
\|\psi\|_{L^\infty_TH^{2n+1}}\lesssim \|\psi_0\|_{H^{2n+1}}\text{exp}\big(C(\|l\|_{L^\infty},\|\frac{1}{l+1}\|_{L^\infty})(\|\psi\|_{X_T}+\|\psi\|_{X_T}^2)).
\label{energyd3}
\end{equation*}
with $C$ a continuous function. Thus the main difficulty of this section will be to prove
$\displaystyle \|I_{\text{quad}}\|_{Y_T}\lesssim \|\psi\|_{X_T}^2$, uniformly in $T$.
Combined with the energy estimate \eqref{energyd3} and similar (easier) bounds for higher order 
terms, this provides global bounds for $\psi$ which imply global well-posedness.\vspace{2mm}\\
In order to perform such estimates we can use integration by part in  \eqref{genericquad} 
either in $s$ or 
$\eta$ (for the relevance of this procedure, see the discussion on space time resonances in the 
introduction). It is thus essential to study where and at which order we have a cancellation of 
$\Omega_{\pm,\pm}(\xi,\eta)=H(\xi)\pm H(\eta)\pm H(\xi-\eta)$ or $\nabla_\eta \Omega_{\pm\pm}$. 
We will denote abusively $H'(\xi)=\frac{2+2|\xi|^2}{\sqrt{2+|\xi|^2}}$ the radial 
derivative of $H$ and note that $\n H(\xi)=H'(\xi)\xi/|\xi|$, we also point out that 
$H'(r)=\frac{2+2r^2}{\sqrt{2+r^2}}$ is stricly increasing.\\
There are several cases that 
have some similarities with the situation for the Schr\"odinger equation, 
see $(\ref{timeR}\ref{spaceR},\ref{totalR})$ for the definition of the resonant sets $\mathcal{T},
\ \mathcal{S}, \ \mathcal{R}$.
\begin{itemize}
 \item $\Omega_{++}=H(\xi)+H(\eta)+H(\xi-\eta)\gtrsim (|\xi|+|\eta|+|\xi-\eta|)(1+
 |\xi|+|\eta|+|\xi-\eta|)$, the time resonant set is reduced to $\mathcal{T}=\{\xi=\eta=0\}$,
 \item $\Omega_{--}=H(\xi)-H(\eta)-H(\xi-\eta)$, we have $\n_\eta \Omega_{--}=H'(\eta)
 \frac{\eta}{|\eta|}+H'(\xi-\eta)\frac{\eta-\xi}{|\eta-\xi|}$. From basic computations 
 \begin{equation*}
  \n_\eta \Omega_{--}=0\Rightarrow 
    \left\{
  \begin{array}{ll}
  H'(\eta)=H'(\xi-\eta)\\ 
  \frac{\xi-\eta}{|\eta-\xi|}=  \frac{\eta}{|\eta|}
  \end{array}\right.
  \Rightarrow 
  \left\{
  \begin{array}{ll}
   |\eta|=|\xi-\eta|\\
   \xi= 2\eta
  \end{array}
\right.
 \end{equation*}
On the other hand $\Omega_{--}(2\eta,\eta)=H(2\eta)-2H(\eta)=0\Leftrightarrow \eta=0$, thus 
$\mathcal{R}=\{\xi=\eta=0\}$.
\item $\Omega_{-+}=H(\xi)-H(\eta)+H(\xi-\eta)$, from similar computations we find that the 
space-time resonant set is $\mathcal{R}=\mathcal{S}=\{\xi=0\}$. The case $\Omega_{+-}$ is 
symmetric.
\end{itemize}
The fact that the space-time resonant set for $\Omega_{+-}$ is not trivial explains why it is quite 
intricate to bound quadratic terms. An other issue pointed out in \cite{GNT3} for their study of 
the Gross-Pitaevskii equation is that the small frequency ``parallel'' resonances are worse than 
for the nonlinear Schr\"odinger equation. Namely near $\xi=\varepsilon\eta$, $\eta<<1$ we have 
\begin{equation*}
H(\varepsilon\eta)-H(\eta)+H((\varepsilon-1)\eta)\sim \frac{-3\varepsilon |\eta|^3}{2\sqrt{2}}
=\frac{-3|\xi|\,|\eta|^2}{2\sqrt{2}},
\text{ while }|\varepsilon\eta|^2-|\eta|^2+|(1-\varepsilon)\eta|^2\sim -2|\eta|\, |\xi|,
\end{equation*}
we see that integrating by parts in time causes twice more loss of derivatives than prescribed by
Coifman-Meyer's theorem, and there is no hope 
even for $\xi/\Omega$ to belong to any standard class of multipliers. Thus it seems unavoidable 
to use the rough multiplier theorem \ref{singmult}.
\subsection{Normal form}\label{secnorm}
In view of the discussion above, the frequency set $\{(\xi,\eta):\ \xi=0\}$ is expected to 
raise some special difficulty. On the other hand the real part of the nonlinearity in 
$(\ref{GPcomp})$ is better behaved than the imaginary part since 
it has the operator $U(\xi)$ in factor whose cancellation near $\xi=0$ should compensate the 
resonances. In the spirit of $\cite{GNT3}$ we will use a 
normal form in order to have a similar cancellation on the imaginary part.
In order to write the nonlinearity as essentially quadratic we set $a'(1)=\alpha$, and rewrite 
\begin{equation}\label{cub1}
\text{Im}(\mathcal{N})(\psi)=-\alpha l\Delta \phi-\nabla \phi\cdot \nabla l+
\big[\big(1+\alpha l-a(1+l)\big)\Delta \phi\big]=
-\alpha l\Delta \phi-\nabla \phi\cdot \nabla l+R.
\end{equation}
From now on, we will use the notation $R$ as a placeholder for remainder terms that should be at 
least cubic. The detailed analysis of $R$ will be provided in section \ref{estimR}.
At the Fourier level, the quadratic terms $-\alpha l\Delta \phi -\n \phi\cdot \n l$ can be written as follows: 
\begin{equation}
\label{acub1}
-\alpha l\Delta \phi -\n \phi\cdot \n l=-\alpha{\rm div}(l\n\phi)+(\alpha-1)\n\phi\cdot\n l.
\end{equation}
We define the change of variables as $l\rightarrow l-B[\phi,\phi]+B[l,l]$, with $B$ a symmetric 
bilinear multiplier to choose later. We have
\begin{equation}
\begin{aligned}
&  \partial_t\big(-B[\phi,\phi]+ B[l,l]\big)=2B[\phi,(-\Delta+2)l]+2B[-\Delta \phi,l]\\
&\hspace{5cm}+2B\big[\phi,\mathcal{N}_1(\phi,l)\big]+2B\big[\mathcal{N}_2(\phi,l),l\big]\\
&=2B[\phi,(-\Delta+2)l]+2B[-\Delta \phi,l]+R,
\end{aligned}
\label{bcub1}
\end{equation}
where the quadratic terms amount to a bilinear Fourier multiplier $B'[\phi,l]$, with symbol 
$B'(\eta,\xi-\eta)=2B(\eta,\xi-\eta)\big(|\eta|^2+2+|\xi-\eta|^2\big)$. The evolution equation on 
$l_1= l-B(\phi,\phi)+B(l,l)$ is using (\ref{acub1}), (\ref{bcub1})
\begin{eqnarray*}
\p_t l_1+\Delta \phi=B''(\phi,l)-\alpha\text{div}(l\n \phi )+R,\\
B''(\eta,\xi-\eta)=2B(\eta,\xi-\eta)(2+|\eta|^2+|\xi-\eta|^2)+(1-\alpha)\eta\cdot (\xi-\eta).
\end{eqnarray*}
The natural choice  is thus to take (note that if $\alpha=1$ the normal form is just the identity)
\begin{equation*}
B(\eta,\xi-\eta)=\frac{(\alpha-1) \eta\cdot (\xi-\eta)}{2+|\eta|^2+|\xi-\eta|^2}.
\end{equation*}
For this choice, we have then:
\begin{equation}
\p_t l_1+\Delta \phi=-\alpha\text{div}(l\n \phi )+R,\\
\end{equation}
In addition from (\ref{systedepart}) we get:
\begin{equation}
\begin{array}{lll}
\displaystyle \partial_t\phi-\Delta l_1+2l_1&=&-\D b(\phi,l)+2 b(\phi,l) +(a(1+l)-1)\Delta l
-\frac{1}{2}\big(|\nabla \phi|^2-|\nabla l|^2\big)\\
&&+(2l-\widetilde{g}(1+l)),
\end{array}
\label{systedepart2}
\end{equation}
with $l_1=l-B[\phi,\phi]+B[l,l]=l+b(\phi,l)$.
Setting $\phi_1=U\phi$ the system becomes:
\begin{eqnarray*}
 \partial_t\phi_1+Hl_1&=&U\bigg(\alpha \,l\Delta l-
 \frac{1}{2}\big(|\nabla U^{-1}\phi_1|^2-|\nabla l|^2\big)+(-\Delta+2)b(\phi,l)-\widetilde{g}''(1) l^2\bigg)
 +R,\\
\partial_tl_1-H \phi_1&=&-\alpha\text{div}(l\nabla \phi)+R.
\end{eqnarray*}
\paragraph{Final form of the equation}
Finally, if we replace in the quadratic terms $l=l_1-b(\phi,l)$ and set $z=\phi_1+il_1$ 
we obtain 
\begin{eqnarray}
\nonumber 
\partial_tz-iHz&=&U\big(\alpha \,l_1\Delta l_1- \frac{1}{2}\big(|\nabla U^{-1}\phi_1|^2
 -|\nabla l_1|^2-\widetilde{g}''(1) l_1^2\big)+(-\Delta+2)b(\phi,l_1)\big)
 -i\alpha\text{div}(l_1\nabla \phi)\\
\nonumber 
&&+U\big(\alpha(-b(\phi,l)\Delta l_1-l_1\Delta b(\phi,l)+b(\phi,l)\Delta b(\phi,l)-2\nabla b(\phi,l)\cdot \nabla l+|\nabla b(\phi,l)|^2\\
\nonumber &&+(-\Delta+2)(-2B[l_1,b(\phi,l)]+B[b(\phi,l),b(\phi,l)])-\widetilde{g}''(1) (b(\phi,l))^2+2\widetilde{g}''(1) l_1 b(\phi,l) \big)    \big)\\
\nonumber &&\hspace{9cm}+i\alpha\text{div}(b(\phi,l)\nabla \phi)
+R\\
\label{EKnormal}&=&Q(z)+R:=\mathcal{N}_z,
\end{eqnarray}
where $Q(z)$ contains the quadratic terms (the first line), $R$ the cubic and quartic terms. 

\begin{rmq}
It is noticeable that this change of unknown is not singular in term of the new variable 
$\phi_1=U \phi$, indeed $B(\phi,\phi)=\widetilde{B}(\n \phi,\n \phi)$ where 
$\widetilde{B}(\eta,\xi-\eta)=\frac{\alpha-1}{(2+|\eta|^2+|\xi-\eta|^2)}$ is smooth, so that 
$B(\phi,\phi)=\widetilde{B}(\n U^{-1}\phi_1,\n U^{-1}\phi_1)$ acts on $\phi_1$ as a composition 
of smooth bilinear and linear multipliers.
\end{rmq}

It remains to check that the normal form is well defined in our functional framework.
We shall also prove that is cancels asymptotically.

\begin{prop}\label{estimformenormale}
For $N>4$, $k\geq 2$, the map $\phi_1+il\mapsto z:=\phi_1+i(l+b(\phi,l))$ is bi-Lipschitz on the 
neighbourhood of $0$ in $X_\infty$, 
Moreover, $\psi=\phi_1+i l$ and $z$ have the same asymptotic as $t\rightarrow \infty$:
\begin{equation*}
\|\psi-z\|_{X(t)}=O(t^{-1/2}).
\end{equation*}

\end{prop}
\begin{proof}
The terms $B[\phi,\phi]$ and $B[l,l]$ are handled in a similar way, we only treat the first case 
which is a bit more involved as we have the singular relation $\phi=U^{-1}\phi_1$.
Note that $B[\phi,\phi]=\widetilde{B}(\n \phi,\n \phi)$, with $\widetilde{B}[\eta,\xi-\eta]=
(\alpha-1)\frac{1}{2+|\eta|^2+|\xi-\eta|^2}$, and $\n U^{-1}=\la \n \ra \circ R_i$ so there is 
no real issue as long as we avoid the $L^\infty$ space.
Also,we split $B=B\chi_{|\eta|\gtrsim|\xi-\eta|}+B(1-\chi_{|\eta|\gtrsim|\xi-\eta|})$ where 
$\chi$ is smooth outside $\eta=\xi=0$, homogeneous of degree $0$, equal to $1$ near 
$\{|\xi-\eta|=0\}\cap \mathbb{S}^{2d-1}$ and $0$ near $\{|\eta|=0\}\cap \mathbb{S}^{2d-1}$. As can 
be seen from the change of variables $\zeta=\xi-\eta$, these terms are symmetric so we can simply 
consider the first case.\\
By interpolation, we have:
\begin{equation}
\forall\,2\leq q\leq p,\ \|\psi\|_{W^{k,q}}\lesssim \|\psi\|_{X(t)}/\ct^{3(1/2-1/q)}.
\label{techj}
\end{equation}
For the $H^N$ estimate we have 
from the Coifman-Meyer theorem (since the symbol $\widetilde{B}$ has the form $\frac{1}{2+|\eta|^2+|\xi-\eta|^2}$), the embedding $H^1\mapsto L^3$ and the boundedness of the Riesz 
multiplier, 
\begin{equation*}
\|B[U^{-1}\phi_1,U^{-1}\phi_1]\|_{H^N}
\lesssim \big\|\n U^{-1}
\phi_1\big\|_{W^{N-2,3}} \big\|\n U^{-1}\phi_1\big\|_{L^{6}}\lesssim \|\phi_1\|_{X(t)}^2/\ct.
\end{equation*}
For the weighted estimate $\|xe^{-itH}B[\phi,\phi]\|_{L^2}$, since $\phi=U^{-1}(\psi
+\overline{\psi})/2$, we have a collection of terms that read in the Fourier variable:
\begin{eqnarray*}
\mathcal{F}\big( xe^{-itH}B[U^{-1}\psi^\pm,U^{-1}\psi^\pm]\big)=
\n_\xi \int e^{-it\Omega_{\pm\pm}}B_1(\eta,\xi-\eta)
\tilde{\psi}^\pm(\eta)\tilde{\psi}^{\pm}(\xi-\eta)d\eta,\\
 \text{ where }B_1=\frac{\eta U^{-1}(\eta)\cdot (\xi-\eta)
 U^{-1}(\xi-\eta)}{2+|\eta|^2+|\xi-\eta|^2}\chi_{|\eta|\gtrsim|\xi-\eta|},\ 
 \Omega_{\pm\pm}=-H(\xi)\mp H(\eta)\mp H(\xi-\eta).
\end{eqnarray*}
If the derivative hits $B_1$, in the worst case it adds a singular term $U^{-1}(\xi-\eta)$, so 
that from the embedding $\dot{H}^1\hookrightarrow L^6$
\begin{eqnarray*}
\bigg\|\int e^{-it\Omega_{\pm\pm}}(\n_\xi B_1)
 \tilde{\psi}^\pm(\eta)\tilde{\psi}^{\pm}(\xi-\eta)d\eta\bigg\|_{L^2} 
=\big\|\n_\xi B_1[\psi^\pm,\psi^{\pm}]\big\|_{L^2} &\lesssim& \|U^{-1}\psi\|_{W^{1,6}}
 \|\psi\|_{W^{1,3}}\\
 &\lesssim&\|\psi\|_{X(t)}^2/\ct^{1/2}.
\end{eqnarray*}
If the derivative hits $\widetilde{\psi}^\pm(\xi-\eta)$ we use the fact that the symbol 
$\frac{\la \xi-\eta\ra^2\chi_{|\eta|\gtrsim |\xi-\eta|}}{2+|\eta|^2+|\xi-\eta|^2}$ is of 
Coifman-Meyer type
\begin{eqnarray*}
\bigg\|\int e^{it\Omega_{\pm\pm}}B_1(\eta,\xi-\eta)
 \tilde{\psi}^\pm(\eta)\n_\xi \tilde{\psi}^{\pm}(\xi-\eta)d\eta\bigg\|_{L^2} 
 &\lesssim& 
\|\la\n\ra\psi \|_{L^{6}}\|\la \n\ra ^{-2}\la\n\ra e^{itH}xe^{-itH}\psi\|_{L^3}\\
 &\lesssim& \|\psi\|_{X(t)}^2/\ct.
\end{eqnarray*}
Finally, if the derivative hits $e^{-it\Omega_{\pm\pm}}$ we note that $\nabla_\xi\Omega_{\pm\pm}
=\nabla_\xi H(\xi)\mp \nabla_\xi H(\xi-\eta)$, where both term are multipliers of order $1$ so
\begin{eqnarray*}
\bigg\|\int e^{it\Omega_{\pm\pm}}it(\n_\xi\Omega_{\pm\pm})B_1
 \tilde{\psi}^\pm(\eta)\tilde{\psi}^{\pm}(\xi-\eta)d\eta\bigg\|_{L^2} 
 &\lesssim &t\|\psi\|_{W^{1,3}}\|\psi\|_{W^{1,6}}\\
 &\lesssim& \|\psi\|_{X(t)}^2/\ct^{1/2}.
\end{eqnarray*}
The $W^{k,p}$ norm is also estimated using the Coifman-Meyer theorem and the boundedness 
of the Riesz multipliers:
\begin{eqnarray*}
\|B_1[\psi^\pm(t),\psi^\pm(t)]\|_{W^{k,p}}\lesssim \|\psi\|_{W^{k-1,1/12-\varepsilon/2}}^2
\lesssim \|\psi\|_{W^{k,1/6-\varepsilon}}^2\lesssim \frac{\|\psi\|_{X(t)}^2}{\ct^{2+6\varepsilon}}.
\end{eqnarray*}
Gluing all the estimates we have proved 
\begin{equation*}
 \|B[U^{-1}\psi,U^{-1}\psi]\|_{X(t)}^2\lesssim \|\psi\|_{X(t)}^2/\ct^{1/2},
 \|B[U^{-1}\psi,U^{-1}\psi]\|_X^2\lesssim \|\psi\|_X^2,
\end{equation*}
thus using the second estimate we obtain from a fixed point argument that the map 
$\phi_1+il\mapsto \phi_1+i(l-B[\phi,\phi]+B[l,l])$
defines a diffeomorphism on a neighbourhood of $0$ in $X$. The first estimate proves the second 
part of the proposition.
\end{proof}
With similar arguments, we can also obtain the following:
\begin{prop}
Let $z_0=U\phi_0+i(l_0-B[\phi_0,\phi_0]+B[l_0,l_0])$, the smallness condition of theorem 
$(\ref{theod4})$ is equivalent to the smallness of 
$\|z_0\|_{H^{2n+1}}+\|xz_0\|_{L^2}+\|z_0\|_{W^{k,p}}$.
\end{prop}

\subsection{Bounds for cubic and quartic nonlinearities}\label{estimR}
Let us first collect the list of terms in $R$ (see $(\ref{cub1}), (\ref{bcub1}), (\ref{EKnormal})$ ) with $b=b(\phi,l)$:
\begin{eqnarray*}
\begin{aligned}
&(1+\alpha l-(a(1+l))\Delta \phi,\ B[\phi,\mathcal{N}_1(\phi,l)],\ B[\mathcal{N}_2(\phi,l),l],\
i\alpha\text{div}(b\nabla \phi),\\
&U\big(\alpha(-b\Delta l_1-l_1\Delta b+b\Delta b-2\nabla b\cdot \nabla l+|\nabla b|^2
(-\Delta+2)b(\phi,-b)-2B[l_1,b]+B[b,b]\big).
\end{aligned}
\end{eqnarray*}
We note that they are all either cubic (for example $B[\phi,|\nabla \phi|^2]$) or quartic (for 
example $B[b,b]$). $B$ is a smooth bilinear multiplier and as we already pointed out, 
$\phi$ always appears with a gradient, we can replace everywhere $\phi$ by $\phi_1=U\phi$ up to 
the addition of Riesz multipliers. \\
Since the estimates are relatively straighforward, we only detail the case of the cubic 
term $B[\phi,|\nabla \phi|^2]$ which comes from $B[\phi,\mathcal{N}_1(\phi)]$ 
(quartic terms are simpler). Since $\phi=U^{-1}(\psi+\overline{\psi})/2$ we are 
reduced to bound in $Y_T$ (see \ref{Fspaces}) terms of the form
\begin{equation*}
I(t)=\int_0^te^{i(t-s)H}B[U^{-1}\psi^\pm,|U^{-1}\nabla \psi^\pm|^2]ds.
\end{equation*}
\begin{prop}
For any $T>0$, we have the a priori estimate 
\begin{equation*}
\sup_{[0,T]}\|I(t)\|_{Y_T}\lesssim \|\psi\|_{X_T}^3.
\end{equation*}

\end{prop}

\begin{proof}\textbf{The weighted bound}\\[2mm]
First let us write 
\begin{eqnarray*}
\begin{aligned}
&xe^{-itH}I(t)=\int_0^te^{-isH}\bigg((-is\nabla_\xi H)B[U^{-1}\psi^\pm,(U^{-1}\nabla \psi^\pm)^2]
+B[U^{-1}\psi^\pm,x(U^{-1}\nabla \psi^\pm)^2]\\
&\hspace{3cm}+\n_\xi B[U^{-1}\psi^\pm,(U^{-1}\nabla \psi^\pm)^2] \bigg)ds,\\
&= I_1(t)+I_2(t)+I_3(t).
\end{aligned}
\end{eqnarray*}
Taking the $L^2$ norm and using the Strichartz estimate with $(p',q')=(2,6/5)$ we get
\begin{eqnarray*}
\|I_1\|_{L^\infty_TL^2}&\lesssim& \|(s\nabla_\xi H)B[U^{-1}\psi^\pm,(U^{-1}\nabla \psi^\pm)^2]\|
_{L^2(L^{6/5})}\\
&\lesssim& \|sB[U^{-1}\psi^\pm,(U^{-1}\nabla \psi^\pm)^2]\|_{L^2(W^{1,6/5})},\\
\|I_2\|_{L^\infty_TL^2}&\lesssim& \|B[U^{-1}\psi^\pm,x(U^{-1}\nabla \psi^\pm)^2]\|
_{L^2(L^{6/5})}.
\end{eqnarray*}
We have then from Coifman-Meyer's theorem, H\"older's inequality, continuity of the Riez operator and (\ref{techj})
\begin{equation}
\label{weightloss}
\begin{aligned}
&\|sB[U^{-1}\psi^\pm,(U^{-1}\nabla \psi^\pm)^2]\|_{L^2_T(W^{1,6/5})}\lesssim 
\big\|s\|\psi\|_{W^{2,6}}^2\|\psi\|_{H^2}\big\|_{L^2_T}\lesssim \|\psi\|_{X_T}^3,\\
&\|I_2\|_{L^\infty_T(L^2)}\lesssim \big\|\|\psi\|_{W^{1,6}}\|x(\nabla 
U^{-1}\psi^\pm)^2\|_{L^{\frac{3}{2}}}\big\|_{L^2_T}.
\end{aligned}
\end{equation}
The loss of derivatives in $I_2$ can be controlled thanks to a paraproduct: let 
$(\chi_j)_{j\geq 0}$ with $\sum \chi_j(\xi)=1$, $\text{supp}(\chi_0)\subset B(0,2),\ 
\text{supp}(\chi_j)\subset \{2^{j-1}\leq \xi \leq 2^{j+1}\},\ j\geq 1$, and set 
$\widehat{\Delta_j\psi}:=\chi_j\widehat{\psi}$, $S_j\psi=\sum_0^j\Delta_k\psi$. Then 
\begin{equation*}
(U^{-1}\nabla \psi^\pm)^2=\sum_{j\geq 0}(\nabla U^{-1}S_j\psi^\pm)(\nabla U^{-1}\Delta_j\psi^\pm)
+\sum_{j\geq 1}(\nabla U^{-1} S_{j-1}\psi^\pm) (\nabla U^{-1}\Delta_j\psi^\pm)
\end{equation*}
For any term of the first scalar product we have
\begin{eqnarray*}
x\big((\partial_kU^{-1}S_j\psi^\pm)(\partial_kU^{-1}\Delta_j\psi^\pm)\big)&=&
(\partial_kU^{-1}S_jx\psi^\pm)(\partial_k U^{-1}\Delta_j \psi^\pm)\\
&&+([x,\partial_kU^{-1}S_j]\psi^\pm) (\partial_kU^{-1}\Delta_j\psi^\pm).
\end{eqnarray*}
From H\"older's inequality, standard commutator estimates, the Besov embedding 
$W^{3,6}\hookrightarrow B^{2}_{6,1}$ and $(\ref{5.5})$ we get 
\begin{eqnarray}
\label{para1}
\sum_j\|(\partial_kU^{-1}S_jx\psi^\pm)(\partial_k U^{-1}\Delta_j \psi^\pm)\|_{L^{3/2}}
\lesssim \sum_j2^j\|x\psi \|_{L^{2}} 2^j\|\Delta_j\psi\|_{L^6}\lesssim \|x\psi\|_{L^2}
\|\psi\|_{W^{3,6}},\\
\label{para2}
\sum_j \|([x,\partial_kU^{-1}S_j]\psi^\pm) (\partial_kU^{-1}\Delta_j\psi^\pm)\|_{L {3/2}}\lesssim 
\|U^{-1}\psi\|_{H^1}\|\psi\|_{W^{1,6}}\lesssim \|\psi\|_{X_T}^2/\ct.
\end{eqnarray}
Moreover, $x\psi=x e^{itH}e^{-itH}\psi=e^{itH}xe^{-itH}\psi+it\nabla_\xi H \psi$ so that : 
$$\|x\psi(t)\|_{L^2}\lesssim \ct \|\psi\|_{X_T}$$
Similar computations can be done for 
$\sum_{j\geq 1}(\nabla U^{-1}S_{j-1} \psi^\pm) (\nabla U^{-1}\Delta_j\psi^\pm)$, finally 
(\ref{para1}), (\ref{para2}) and (\ref{techj}) imply 
\begin{equation*}
\|x(U^{-1}\nabla \psi^\pm)^2\|_{L^{3/2}}\lesssim \|\psi\|_{X_T}^2.
\end{equation*}
Plugging the last inequality in $(\ref{weightloss})$ we can conclude
\begin{equation*}
\|I_2\|_{L^\infty_TL^2}\lesssim \big\|\|\psi\|_{X_T}^3/\ct\|_{L^2_T}\lesssim \|\psi\|_{X_T}^3.
\end{equation*}
\textbf{The $W^{k,p}$ decay} We can apply the dispersion estimate in the same way as 
in section $\ref{sectiondlarge}$:
\begin{eqnarray}
\nonumber
\bigg\|\int_0^{t-1}e^{i(t-s)H}B[U^{-1}\psi^\pm,\,(U^{-1}\nabla \psi^\pm)^2]ds\bigg\|_{W^{k,p}}
&\lesssim &\int_0^{t-1}\frac{\|B[U^{-1}\psi^\pm,\,(U^{-1}\nabla \psi^\pm)^2]\|_{W^{k,p'}}}
{(t-s)^{1+3\varepsilon}}ds\\
\nonumber
&\lesssim& \int_0^{t-1}\frac{\|\nabla U^{-1}\psi\|_{W^{k,3p'}}^3}{(t-s)^{1+3\varepsilon}}\\
\label{intermcub} &\lesssim& \int_0^{t-1}\frac{\|\psi\|_{W^{k+1,3p'}}^3}{(t-s)^{1+3\varepsilon}}
\end{eqnarray}
We then use interpolation and the estimate (\ref{techj}) with $q=p'$, we have then:
\begin{equation*}
\|\psi\|_{W^{k+1,3p'}}\lesssim \|\psi\|_{W^{k,3p'}}^{(J-1)/J}\|\psi\|_{W^{k+J,3p'}}^{1/J},
\|\psi(t)\|_{W^{k,3p'}}\lesssim \frac{\|\psi\|_{X_T}}{(1+t)^{2/3-\varepsilon}}.
\end{equation*}
Since $3p'<6$, we have $\|\psi\|_{W^{k+J,3p'}}\lesssim \|\psi\|_{H^{k+J+1}}$ by Sobolev embedding, so that for 
$\varepsilon$ small enough, $J$ large enough such that $(2-3\e)(1-\frac{1}{J})\geq 1+3\e$ (but $J\leq N-k-1$) we observe that:
$$\|\psi\|^3_{W^{k+1,3p'}}\lesssim 
\frac{\|\psi\|^3_{X_T}}{\ct^{1+3\e}}$$
 Plugging this inequality in $(\ref{intermcub})$ we conclude that:
$$
\int_0^{t-1}\frac{\|\psi\|_{W^{k+1,3p'}}^3}{(t-s)^{1+3\varepsilon}}\lesssim \frac{\|\psi\|_{X_T}^3}{
\ct^{1+3\varepsilon}}.
$$
For the integral on $[t-1,t]$ it suffices to bound $\|\int_{t-1}^te^{i(t-s)H}B[U^{-1}\psi^\pm,
(U^{-1}\nabla \psi^\pm)^2]ds\|_{W^{k,p}}\lesssim \|\int_{t-1}^t\|B[U^{-1}\psi^\pm,
(U^{-1}\nabla \psi^\pm)^2]ds\|_{H^{k+2}}$ and follow the argument of the proof of proposition 
$\ref{decay}$.
\end{proof}

\section[Quadratic nonlinearities, end of proof]{Bounds for quadratic nonlinearities in 
dimension 3, end of proof}
\label{estimquad}
The following proposition will be repeatedly used (see proposition 4.6 \cite{AudHasp} or 
\cite{GNT3}).
\begin{prop}\label{controlX}
We have the following estimates with $0\leq\theta\leq 1$:
\begin{equation}
\|\psi(t)\|_{\dot{H}^{-1}}\lesssim \|\psi(t)\|_{X(t)},
\label{5.5}
\end{equation}
\begin{equation}
\|U^{-2}\psi\|_{L^6}\lesssim  \|\psi(t)\|_{X(t)}
\label{5.6}
\end{equation}
\begin{equation}
\begin{aligned}
&\||\n|^{-2+\frac{5\theta}{3}}\psi_{<1}(t)\|_{L^6}\lesssim\min(1,t^{-\theta})\|\psi(t)\|_{X(t)},\\
&\||\n|^{\theta}\psi_{\geq1}(t)\|_{L^6}\lesssim\min(t^{-\theta},t^{-1})\|\psi(t)\|_{X(t)}.
\end{aligned}
\label{5.8}
\end{equation}
\begin{equation}
\begin{aligned}
&\|U^{-1}\psi(t)\|_{L^6}\lesssim \la t\ra^{-\frac{3}{5}}\|\psi(t)\|_{X(t)},\\
\end{aligned}
\label{5.9}
\end{equation}
\label{gain}
\end{prop}

In this section, we will assume $\|\psi\|_{X_T}<<1$, for the only reason that
\begin{equation*}
\forall\,m>2,\ \|\psi\|_{X_T}^2+\|\psi\|_{X_T}^m\lesssim \|\psi\|_{X_T}^2.
\end{equation*}
All computations that follow can be done without any smallness assumption, but they would require 
to always add in the end some $\|\psi\|_{X_T}^m$, that we avoid for conciseness.

\subsection{\texorpdfstring{The $L^p$ decay}{decay}}\label{declp}
We now prove decay for the quadratic terms in $(\ref{EKnormal})$, namely
$$\ct^{1+3\varepsilon}\|\int^{t}_{0}e^{i(t-s)H}Q(z)(s)ds\|_{W^{k,p}}\lesssim \|z\|_{X_T}^2.$$
For $t\leq 1$, the estimate is a simple consequence of the product estimate 
$\|Q(z)\|_{H^{k+2}}\lesssim \|z\|_{H^N}^2$
and the boundedness of $e^{itH}:H^s\mapsto H^s$. Thus we focus on the case $t\geq 1$ and note that 
it is sufficient to bound $t^{1+3\ve}\|\int^{t}_{0}e^{i(t-s)H}Q(z)(s)ds\|_{W^{k,p}}$.\\
We recall that the quadratic terms have the following structure (see \eqref{EKnormal})
\begin{equation}\label{strucquad}
Q(z)=U\big(\alpha \,l_1\Delta l_1- \frac{1}{2}\big(|\nabla U^{-1}\phi_1|^2
 -|\nabla l_1|^2  -\widetilde{g}''(1) l_1^2\big)+(-\Delta+2)b(\phi,l_1)\big)
 -i\alpha\text{div}(l_1\nabla U^{-1}\phi_1),
\end{equation}
where $b=-B[\phi,\phi]+B[l_1,l_1],\ B(\eta,\xi-\eta)=\frac{(\alpha-1)\eta\cdot (\xi-\eta)}{2+|\eta|^2
+|\xi-\eta|^2}$ so that any term in $Q$ is of the form 
$(U\circ B_j)[z^\pm,z^\pm],\ j=1\cdots 5$ where $B_j$ 
satisfies $B_j(\eta,\xi-\eta)\lesssim 2+|\eta|^2+|\xi-\eta|^2$.
\subsubsection{Splitting of the phase space}
We split the phase space $(\eta,\xi)$ in non time resonant and non space resonant sets: 
let $(\chi^a)_{a\in 2^\Z}$ standard dyadic partition of unity: $\chi^a \geq 0$, 
$\text{supp}(\chi^a)\subset \{|\xi|\sim a\}$,
$\forall\,\xi\in \R^3\setminus \{0\},\ \sum_{a} \chi^a(\xi)=1$. We define the frequency localized 
symbol $B_j^{a,b,c}=\chi^a(\xi)\chi^b(\eta)\chi^c(\zeta)B_j$.\\
Note that due to the relation 
$\xi=\eta+\zeta$, we have only to consider $B_j^{a,b,c} $ when $a\lesssim b\sim c,\ b\lesssim c\sim a$ or 
$c\lesssim a\sim b$.
We will define in the appendix two disjoint sets of indices $\mathcal{NT},\mathcal{NS}$ such that 
$\mathcal{NT}\cup \mathcal{NS}=\Z^3$ and which correspond, in a sense precised by lemma 
\ref{bacrucial1},\ref{bacrucial} to non time resonant and non space resonant frequencies. 
Provided such sets have been constructed, we write
\begin{eqnarray*}
\sum_{a,b,c}\int^{t}_{0}e^{i(t-s)H}UB_j^{a,b,c}[z^\pm,z^\pm](s)ds&=&
\int^{t}_{0}e^{i(t-s)H}\sum_{a,b,c\in \mathcal{NT}}UB_j^{a,b,c,T}+
\sum_{a,b,c\in \mathcal{NS}}UB_j^{a,b,c,X}ds\\
&:=&\sum_{a,b,c\in \mathcal{NT}}I^{a,b,c,T}+
\sum_{a,b,c\in \mathcal{NS}}I^{a,b,c,X}
\end{eqnarray*}
For $(a,b,c)\in \mathcal{NT}$ (resp. $\mathcal{NS}$) we will use an integration by parts
in time (resp. in the ``space'' variable $\eta$).
\subsubsection{Control of non time resonant terms}\label{secNRT}
The generic frequency localized 
quadratic term is
\begin{eqnarray}
&e^{itH(\xi)}
\displaystyle\int_0^t\int_{\R^d} \bigg(e^{-is(H(\xi)\mp H(\eta)\mp 
H(\xi-\eta)}U(\xi)B_j^{a,b,c,T}(\eta,\xi-\eta)\widetilde{z^\pm}(s,\eta)
\widetilde{z^\pm}(s,\xi-\eta)\biggl)d\eta \, ds
\end{eqnarray}
Regardless of the $\pm$, we set $\Omega=H(\xi)\mp H(\eta)\mp H(\xi-\eta)$. An 
integration by part in $s$ gives using the fact that $e^{-i s\Omega}=\frac{-1}{i\Omega} \p_s(e^{i s\Omega} )$ and $\p_s
\widetilde{z^\pm}(\eta)=e^{ \mp i s  H(\eta)} (\mathcal{N}_z)^\pm(\eta)$,  $\p_s
\widetilde{z^\pm}(\xi-\eta)= e^{ \mp i s  H(\xi-\eta)} (\mathcal{N}_z)^\pm(\xi-\eta) $:
\begin{equation}
\begin{aligned}
I^{a,b,c,T}
=&{\cal F}^{-1}(e^{itH(\xi)}\biggl(
\int_0^{t} \int_{\R^N}  \bigg(\frac{1}{i\Omega} e^{-is \Omega} U(\xi)B^{a,b,c,T}_j(\eta,\xi-\eta)\p_s
\big(\widetilde{z^\pm}(\eta) \widetilde{z^\pm}(\xi-\eta)\big)\bigg)d\eta ds\biggl)\\
&-\biggl[{\cal F}^{-1}(e^{itH(\xi)}\biggl(
\int_{\R^N}  \bigg(\frac{1}{i\Omega} e^{-is \Omega(\xi,\eta)} U(\xi)B^{a,b,c,T}_j(\eta,\xi-\eta)
\big(\widetilde{z^\pm}(\eta) \widetilde{z^\pm}(\xi-\eta)\big)\bigg)d\eta ds\biggl)
\biggl]^{t}_0\\
&=\int_0^{t}e^{i(t-s)H}\bigg(\mathcal{B}^{a,b,c,T}_3[(\mathcal{N}_z)^\pm,z^\pm]
+\mathcal{B}^{a,b,c,T}_3[z^\pm,\,(\mathcal{N}_z)^\pm]\bigg)ds\\
\label{IPPtemps}&\hspace{2cm}-\big[e^{i(t-s) H}\mathcal{B}^{a,b,c,T}_3[z^\pm,z^\pm]\big]_0^{t},
\end{aligned}
\end{equation}
with 
$\displaystyle 
\mathcal{B}^{a,b,c,T}_3(\eta,\xi-\eta)=\frac{U(\xi)}{i\Omega}\chi^{a}(\xi)\chi^b(\eta)
\chi^c (\xi-\eta)B_{j}(\eta,\xi-\eta)$.
\\
In order to use the rough multiplier estimate from theorem \ref{singmult}, we need to control 
$\mathcal{B}_3^{a,b,c,T}$. 
The following lemma extends to our settings the crucial multiplier estimates from \cite{GNT3}.
\begin{lemma}Let $m=\min(a,b,c),\ M=\max(a,b,c),\ l=\min(b,c)$.
For $0<s<2$, we have 
\begin{equation}
\text{if }M\gtrsim 1,\ 
\|\mathcal{B}_3^{a,b,c,T}\|_{[B^{s}]}\lesssim \frac{\la M\ra l^{\frac{3}{2}-s}}{\la a\ra},
\text{ if }M<<1,\ \|\mathcal{B}_3^{a,b,c,T}\|_{[B^{s}]}\lesssim l^{1/2-s}M^{-s}.
\label{10.43}
\end{equation}
\label{bacrucial1}
\end{lemma}
\noindent We postpone the proof to the appendix.
\begin{rmq}
We treat differently $M$ small and $M$ large since we have a loss of derivative on the symbol in low frequencies. Let us mention that the estimate \eqref{10.43} can be written simply as follows:\begin{equation*}
\|\mathcal{B}_3^{a,b,c,T}\|_{[B^{s}]}\lesssim \frac{\la M\ra \la l\ra l^{\frac{1}{2}-s}U(M)^{-s}}
{\la a\ra}
\end{equation*}

\end{rmq}
Lets us start by estimating the first term in $(\ref{IPPtemps})$: we split the time integral 
between $[0,t-1]$ and $[t-1,t]$. 
The sum over $a,b,c$ involves three cases: $b\lesssim a\sim c,\ c\lesssim a\sim b$ and 
$a\lesssim b\sim c$. 
\subparagraph{The case \mathversion{bold}{$b\lesssim a\sim c$}:} for $k_1\in[0,k]$ we have from 
theorem $\ref{singmult}$ with $\sigma=1+3\ve$:
\begin{equation}\label{ipptemps}
\begin{aligned}
&\|\n^{k_1}\int_0^{t-1}e^{i(t-s)H}\sum_{b\lesssim a\sim c} \mathcal{B}^{a,b,c,T}_3
[\mathcal{N}_z^\pm,z^\pm]ds\|_{L^p} \\
&\lesssim \int_0^{t-1}\frac{1}{(t-s)^{1+3\ve}} \sum_{b\lesssim a\sim c}\la a\ra^{k_1} 
\| \mathcal{B}^{a,b,c,T}_3[\mathcal{N}_z^\pm,z^\pm]\|_{L^{p'}} ds,\\
&\lesssim  \int_0^{t-1}\frac{1}{(t-s)^{1+3\ve}}\biggl( \sum_{b\lesssim a\sim c\lesssim 1}
ab\|\mathcal{B}^{a,b,c,T}_3\|_{[B^{\sigma}]} \|U^{-1}
Q(z)\|_{L^{2}}\|U^{-1}z\|_{L^{2}}\\
&\hspace{3cm}+  \sum_{b\lesssim a\sim c, 1\lesssim a} \la c\ra^{-N+k} U(b)
\|\mathcal{B}^{a,b,c,T}_3\|_{[B^{\sigma}]} \|U^{-1}Q(z)\|_{L^{2}}\|\la \n\ra^N z\|_{L^{2}}
\biggl)  ds+\mathcal{R}
\end{aligned}
\end{equation}
where $\displaystyle 
\mathcal{R}=\int_0^{t-1}\frac{1}{(t-s)^{1+3\ve}} \sum_{b\lesssim a\sim c}\la a\ra^{k_1} 
\|\mathcal{B}^{a,b,c,T}_3[R^\pm,z^\pm]\|_{L^{p'}} ds$.
Using lemma $\ref{bacrucial1}$ we have, provided $\e<\frac{1}{12}$ and $N-k-\frac{1}{2}+3\e>0$: 
$$
\begin{aligned}
&\sum_{b\lesssim a\sim c\lesssim 1}ab\|\mathcal{B}^{a,b,c,T}_3\|_{[B^{\sigma}]}\lesssim 
\sum_{ a\lesssim 1}\sum_{b\lesssim a} ab b^{1/2-1-3\varepsilon}a^{-1-3\varepsilon}\lesssim 
\sum_{a\lesssim 1}a^{1/2-6\varepsilon}\lesssim 1,\\
&\sum_{b\lesssim a\sim c,\ a\gtrsim 1}U(b)\la c\ra^{-N+k}\|\mathcal{B}^{a,b,c,T}_3\|_{[B^{\sigma}]}
\lesssim \sum_{ a\gtrsim 1}\sum_{b\lesssim a} U(b) \frac{ b^{\frac{1}{2}-3\e}}{a^{N-k}}
\lesssim  \sum_{ a\gtrsim 1} \frac{ 1}{a^{N-k}}+ \sum_{ a\gtrsim 1} \frac{ 1}{a^{N-k-\frac{1}{2}+3\e}}\lesssim 1.
\end{aligned}
$$
Using the gradient structure of $Q(z)$ (see \ref{EKnormal}) :
\begin{equation}
\begin{aligned}
& \|U^{-1}Q(z)\|_{L^{2}}\lesssim \|z\|_{W^{2,4}}^2
\lesssim\|z\|_{W^{2,6}}^{\frac{3}{2}}\|z\|_{H^2}^{\frac{1}{2}},
\end{aligned}
\label{decay1}
\end{equation}
so that if we combine these estimates with $\eqref{5.5}$, we get
$$
\begin{aligned}
\|\n^{k_1}\int_0^{t-1}e^{i(s-t)H}\sum_{b\lesssim a\sim c} 
\mathcal{B}^{a,b,c,T}_3[Q(z)^\pm,z]ds\|_{L^p}&\lesssim\|z\|_{X}^3 \int_0^{t-1}
\frac{1}{(t-s)^{1+3\ve}}\frac{1}{\cs^{\frac{3}{2}} }ds\\
&\lesssim\frac{\|z\|_{X}^3}{t^{1+3\ve}}.
\end{aligned}
$$
We bound now $\mathcal{R}$ from $(\ref{ipptemps})$: contrary to the quadratic terms, cubic 
terms have no gradient structure, however 
the nonlinearity is so strong that we can simply use $\|1_{|\eta|\lesssim 1}U^{-1}R\|_2
\lesssim \|R\|_{L^{6/5}}$. Using the same computations as for quadratic terms we get 
$$
\begin{aligned}
&\|\n^{k_1}\int_0^{t-1}e^{i(t-s)H}\sum_{b\lesssim a\sim c} 
\mathcal{B}^{a,b,c,T}_3[R,z^\pm]ds\|_{L^p}\\
&\hspace{2cm} \lesssim  \int_0^{t-1}\frac{1}{(t-s)^{1+3\ve}}\biggl( 
 \|1_{\{|\eta|\lesssim1\}}U^{-1}R\|_{L^{2}}\|U^{-1}z\|_{L^{2}}
 +\|U^{-1}R\|_{L^{2}}\|\la \n\ra^N z\|_{L^{2}}\biggl)  ds.
\end{aligned}
$$
According to $(\ref{EKnormal})$ the cubic terms involve only smooth multipliers
and do not contain derivatives of order larger than $2$, thus we can generically treat them like
$(\la \n \ra^2z)^3$ using the proposition \ref{estimformenormale}; we have then:
\begin{eqnarray*}
\|R\|_{L^{6/5}}\lesssim \|z\|_{H^2}\|z\|_{W^{2,6}}^2\lesssim \frac{\|z\|_X^3}{
\ct^2},\
\|R\|_{L^2}\lesssim \|z\|_{W^{2,6}}^3\lesssim \frac{\|z\|_X^3}{\ct^2}.
\end{eqnarray*}
This closes the estimate as $\displaystyle \int_0^{t-1}\frac{1}{(t-s)^{1+3\varepsilon}\cs^2}ds
\lesssim \frac{1}{t^{1+3\varepsilon}}$. We proceed similarly for the quartic terms.
\\
It remains to deal with the term $\int^t_{t-1}$, using Sobolev embedding we have:
$$
\begin{aligned}
&\|\n^{k_1}\int_{t-1}^{t}e^{i(t-s)H}\sum_{b\lesssim a\sim c} \mathcal{B}^{a,b,c,T}_3
[\mathcal{N}_z^\pm,z^\pm]ds\|_{L^p}\lesssim \int_{t-1}^{t}\|(\cdots )\|_{H^{k_2}}ds,
\end{aligned}
$$
with $k_2=k+1+3\ve$. Again, with $\sigma=1+3\varepsilon$ we get using theorem \ref{singmult} and Sobolev embedding:
$$
\begin{aligned}
&\|\n^{k_1}\int_{t-1}^{t}e^{i(t-s)H}\sum_{b\lesssim a\sim c} \mathcal{B}^{a,b,c,T}_3
[\mathcal{N}_z^\pm,z^\pm]ds\|_{L^p}\lesssim \int_{t-1}^{t}\|\sum_{b\lesssim a\sim c} 
\mathcal{B}^{a,b,c,T}_3[\mathcal{N}_z^\pm,z^\pm]\|_{H^{k_2}} ds\\
&\lesssim \int_{t-1}^{t}\big(\sum_{b\lesssim a\sim c\lesssim 1} ab\|\mathcal{B}^{a,b,c,T}_3
\|_{[B^\sigma]} \|U^{-1}Q\|_{L^{2}}\|U^{-1}z\|_{L^p} \\
&\hspace{3cm}+\sum_{b\lesssim a\sim c, 1\lesssim a} U(b)a^{k_2-(N-1-3\e)} \|\mathcal{B}^{a,b,c,T}_3\|_{[B^\sigma]}
\|U^{-1}Q\|_{L^{2}}\|\la\n\ra^Nz\|_{L^2}\big) ds+\mathcal{R},\\[2mm]
\end{aligned}
$$
where $\mathcal{R}$ contains higher order terms that are easily controlled.
Using $\|U^{-1}z\|_{L^p}\lesssim \|z\|_{H^2}$ and the same estimates as previously, we 
can conclude provided that $N$ is sufficiently large:
$$
\|\n^{k_1}\int_{t-1}^{t}e^{i(t-s)H}\sum_{b\lesssim a\sim c} \mathcal{B}^{a,b,c,T}_3
[\mathcal{N}_z^\pm,z^\pm] ds\|_{L^{p}}\lesssim \|u\|_{X}^3 \int_{t-1}^{t}
\frac{1}{\cs^{3/2}} ds \lesssim\frac{\|z\|_{X}^3}{t^{1+3\ve}}. 
$$
\subparagraph{The case \mathversion{bold}{$c\lesssim a\sim b$}} As for 
$b\lesssim a\sim c$ we start with
$$
\begin{aligned}
&\|\n^{k_1}\int_1^{t-1}e^{i(t-s)H}\sum_{c\lesssim a\sim b} \mathcal{B}^{a,b,c,T}_3
[\mathcal{N}_z^\pm,z^\pm]ds\|_{L^p}\\
&\lesssim  \int_1^{t-1}\frac{1}{(t-s)^{1+3\ve}}\biggl( \sum_{c\lesssim a\sim b\lesssim 1}bc
\|\mathcal{B}^{a,b,c,T}_3\|_{[B^{\sigma}]}\|U^{-1}Q(z)\|_{L^{2}}\|U^{-1}z\|_{L^{2}}\\
&\hspace{3cm}+  \sum_{c\lesssim a\sim b, 1\lesssim a} 
\la b\ra^{-1}\|\mathcal{B}^{a,b,c,T}_3\|_{[B^{\sigma}]}
\|\la\n\ra^{k+1} Q(z)\|_{L^{2}}\|z\|_{L^{2}}\biggl)  ds+
 \mathcal{R}.
\end{aligned}
$$
with $\sigma=1+3\ve$ and $\mathcal{R}$ contains the other nonlinear terms (which, again, we will 
not detail). This case is symmetric
from $b\lesssim a\sim c$ except for the term $\|\la \nabla\ra ^{k+1}Q(z)\|_{L^2}$, which is 
 estimated as follows.
Let $1/q=1/3+\varepsilon$, $k_3=\frac{1}{2}-3\ve$. If $k+2+k_3\leq N$ then using the structure of $Q$ 
(see \eqref{strucquad}) and Gagliardo Nirenberg inequalities we get:
$$
\begin{aligned}
 \|\la \n\ra ^{k+1} Q(z)\|_{L^{2}}\lesssim 
 \|z\|_{W^{2,p}}\|z\|_{W^{k+3,q}}\lesssim \|z\|_{W^{2,p}}\|z\|_{H^{k+3+k_3}}
&\lesssim \|z\|_{X}^2/\ct^{1+3\ve},
\end{aligned}
$$
Using the multiplier bounds as for the case $b\lesssim a\sim c$, we obtain via the lemma \ref{bacrucial1}:
$$
\begin{aligned}
\|\n^{k_1}\int_0^{t-1}e^{i(t-s)H}\sum_{c\lesssim a\sim b} \mathcal{B}^{a,b,c,T}_3
[\mathcal{N}_z^\pm,z^\pm]ds\|_{L^{p}}
\lesssim &\|z\|_{X}^3 \int_{0}^{t-1}\frac{1}{(t-s)^{1+3\ve}}\frac{1}{\cs^{(1+3\ve)}}
ds\\
\lesssim&\frac{\|z\|_{X}^3}{t^{1+3\ve}}.
\end{aligned}
$$
The bound for the integral on $[t-1,t]$ is obtained by similar arguments.
\subparagraph{The case \mathversion{bold}{$a\lesssim b\sim c$}} We have using theorem \ref{singmult} and the fact that 
the support of ${\cal F}(\sum_{a\lesssim b}a^{k_1} 
\mathcal{B}^{a,b,c,T}_3[\mathcal{N}_z^\pm,z^\pm])$ is localized in a ball $B(0,b)$  :
$$
\begin{aligned}
&\|\n^{k_1}\int_0^{t-1}e^{i(t-s)H}\sum_{a\lesssim b\sim c} \mathcal{B}^{a,b,c,T}_3
[\mathcal{N}_z^\pm,z^\pm]ds\|_{L^p}\\
&\lesssim \int_0^{t-1}\frac{1}{(t-s)^{1+3\ve}}\|\sum_{a\lesssim b \sim c}a^{k_1} 
\mathcal{B}^{a,b,c,T}_3[\mathcal{N}_z^\pm,z^\pm]\|_{L^{p'}}ds\\
&\lesssim \int_0^{t-1}\frac{1}{(t-s)^{1+3\ve}} \sum_{b\sim c}\frac{1}{\la b\ra^{N-2}}U(b)U(c)\|
\sum_{a\lesssim b}\la a\ra^{k}  \mathcal{B}^{a,b,c,T}_3\|_{[B^\sigma]}
\|U^{-1}Q(z)\|_{L^2}\|U^{-1}\la\n\ra^Nz\|_{L^{2}} ds\\
&\hspace{1cm} +\mathcal{R},
\end{aligned}
$$
where as previously, $\mathcal{R}$ is a remainder of higher order terms that are not difficult to 
bound. We observe that for any symbols $(B^a(\xi,\eta))$ such that 
$$\forall\,\eta,\ |a_1-a_2|\geq 2\Rightarrow 
\text{supp}(B^{a_1}(\cdot,\eta))\cap \text{supp}(B^{a_2}(\cdot,\eta))=\emptyset,$$
then
\begin{equation}\label{tricksuma}
\|\sum_a B^{a}\|_{[B^\sigma]}\lesssim \sup_a \|B^a\|_{[B^\sigma]}.
\end{equation}
This implies using lemma \ref{bacrucial1} and provided that $N$ is large enough:
$$
\begin{aligned}
\sum_{b\sim c }\frac{1}{\la b\ra^{N-2}} U(b)U(c)\|\sum_{a\lesssim b}\la a\ra^{k}
\mathcal{B}^{a,b,c,T}_3\|_{[B^\sigma]} 
&\lesssim \sum_{b }\frac{1}{\la b\ra^{N-2}} U(b)^2\sup_{a\lesssim b} \la a\ra^{k}
\frac{ b^{\frac{1}{2}-\sigma}U(M)^{-\sigma}\la b\ra\la M\ra}{\la a\ra }  \\
&\lesssim \sum_b\frac{U(b)^{5/2-2\sigma} }{\la b\ra^{N+\sigma-k-7/2}}\lesssim 1.
\end{aligned}
$$
We have finally using (\ref{decay1}):
$$
\begin{aligned}
\|\n^{k_1}\int_0^{t-1}e^{i(t-s)H}\sum_{a\lesssim b\sim c} \mathcal{B}^{a,b,c,T}_3
[\mathcal{N}_z^\pm,z^\pm]ds\|_{L^{p}}&\lesssim  \|z\|_{X}^3 \int_{0}^{t-1}\frac{1}{(t-s)^{1+3\ve}}
\frac{1}{\cs^{3/2}}ds\\
&\lesssim\frac{ \|u\|_{X}^3}{t^{1+3\varepsilon}}.
\end{aligned}
$$
We proceed in a similar way to deal with the integral on $[t-1,t]$. This end the estimate for the first term in \eqref{IPPtemps}.\vspace{2mm}\\
The second term is symmetric from the first, it remains to deal with the 
boundary term: 
$\displaystyle \|\n^{k_1}\big[e^{i(t-s)H}
\mathcal{B}^{a,b,c,T}_3[z^\pm,z^\pm]\big]_0^t\|_{L^{p}}.$
We have: 
\begin{equation}
\begin{aligned}
\|\big[\n^{k_1}e^{i(t-s)H}\mathcal{B}^{a,b,c,T}_3[z^\pm,z^\pm]\big]_0^t\|_{L^{p}}
\leq &\|\n^{k_1}e^{-itH}\mathcal{B}^{a,b,c,T}_3[z^\pm_0,z^\pm_0]\|_{L^{p}}\\
&+\|\n^{k_1}\mathcal{B}^{a,b,c,T}_3[z^\pm(t),z^\pm(t)]\|_{L^{p}}
\end{aligned}
\label{inttemps}
\end{equation}
The first term on the right hand-side 
of (\ref{inttemps}) is easy to deal with using the dispersive estimates of the theorem \ref{dispersion}. For the second term we focus on the case
$b\lesssim a\sim c$, the other areas can be treated in a similar way. Using proposition 
\ref{gain}, Sobolev embedding and the rough multiplier theorem \ref{singmult} with $s=1+3\e$, $q_1=q_2=q_3=p$ (which verifies $2\leq p=\frac{6}{3-2\e}$) we have:
$$
\begin{aligned}
&\sum_{b\lesssim a\sim c\lesssim 1}\|\n^{k_1}\mathcal{B}^{a,b,c,T}_3[z^\pm(t),z^\pm(t)]\|_{L^{p}}
\lesssim \displaystyle 
\sum_{b\lesssim a\sim c}b^{-\frac{1}{2}-3\e}a^{-1-3\e}U(b)U ( c)\|U^{-1}z\|_{L^{p}}^2\\
&\hspace{6cm}\lesssim \sum_{b\lesssim a\sim c}b^{-\frac{1}{2}-3\e}a^{-1-3\e}U(b)U( c)\|U^{-1+3\e}z\|_{L^{6}}^2
\lesssim \frac{\|z\|_X^2}{\ct^{\frac{6}{5}+6\e}},\\
&\sum_{b\lesssim a\sim c,\ a\gtrsim 1}\|\n^{k_1}\mathcal{B}^{a,b,c,T}_3[z^\pm(t),z^\pm(t)]\|_{L^{p}}
\lesssim \displaystyle 
\sum_{b\lesssim a\sim c,\ a\gtrsim 1}\frac{\la a\ra^{k_1}b^{1/2-3\e}}{\la a\ra^{k_1+1}}
\|z\|_{L^{p}}\|z\|_{W^{k_1+1,p}}
\lesssim \frac{\|z\|_X^2}{\ct^{\frac{3}{2}(1+\e)}}
\end{aligned}
$$
where in the last inequality we also used $\|z\|_{W^{k+1,6}}^2\lesssim \|z\|_{W^{k,p}}
\|z\|_{W^{k+2,p}}\lesssim \|z\|_{W^{k,p}}\|z\|_{H^N}$.

\subsubsection{Non space resonance}
\label{secNRS}
In this section we treat the term $\sum_{a,b,c}I^{a,b,c,X}$.
Since control for $t$ small just follows from the $H^N$ bounds, we focus on $t\geq 1$, and 
first note that the integral over $[0,1]\cup [t-1,t]$ is easy to estimate.\vspace{3mm}\\
\textbf{Bounds for $(\int_0^1+\int^t_{t-1})e^{i(t-s)H}Q(z)ds$}\vspace{2mm}\\
In order to estimate $\displaystyle \|\n ^{k_1}\int^t_{t-1}e^{i(t-s)H}Q(z)ds
\|_{L^{p}},$
with $k_1\in[0,k]$ we can simply use Sobolev's embedding ($H^{k+2}\h W^{k,p}$, $H^{N}\h W^{k+4,q}$) and a Gagliardo-Nirenberg type 
inequality (\ref{GN3}) with $\frac{1}{2}=\frac{1}{q}+\frac{1}{p}$ :
$$
\begin{aligned}
\|\int^t_{t-1}\nabla^{k_1}e^{i(t-s)H}Q(z)ds\|_{L^{p}}
&\lesssim  \int^t_{t-1}\|Q(z)\|_{H^{k+2}}ds\\
&\lesssim  \int^t_{t-1}\|z\|_{W^{k+4,q}}\|z\|_{W^{k,p}}ds\\
&\lesssim  \|z\|_{X}^2\int^t_{t-1} \frac{1}{\cs^{1+3\varepsilon}}ds
\lesssim \frac{\|z\|_{X}^2}{\ct^{1+3\varepsilon}}.
\end{aligned}
$$
The estimate on $[0,1]$ follows from similar computations using Minkowski's inequality and the 
dispersion estimate from theorem $\ref{dispersion}$.\vspace{2mm}\\
\textbf{Frequency splitting}\vspace{2mm}\\
Since we only control $xe^{-itH}z$ in $L^\infty L^2$, in order to handle the loss of derivatives 
we follow the idea from \cite{GMS2} which corresponds to distinguish
 low and high frequencies with a threshold frequency depending on $t$. Let $\theta\in C_c^\infty (\R^+)$, 
$\theta|_{[0,1]}=1,\ 
\text{supp}(\theta)\subset[0,2]$, $\Theta(t)=\theta(\frac{|D|}{t^\delta})$, for any quadratic term 
$B_j[z,z]$, we write
\begin{equation*}
B_j[z^\pm,z^\pm]=\overbrace{B_j[(1-\Theta(t))z^\pm,z^\pm]
+B_j[\Theta(t)z^\pm,(1-\Theta)(t)z^\pm]}
^{\text{high frequencies}}
+\overbrace{B_j[\Theta(t)z^\pm,\Theta(t)z^\pm]}^{\text{low frequencies}}.
\end{equation*}
\subsubsection*{High frequencies}
Using the dispersion theorem \ref{dispersion}, Gagliardo-Nirenberg estimate (\ref{GN3}) and Sobolev embedding we have for $\frac{1}{p_1}=\frac{1}{3}+\ve$ and for any quadratic term of $Q$ writing under the form $U B_j[z^\pm,z^\pm]$:
\begin{equation}
\begin{aligned}
&\bigg\|\int^{t-1}_{1}e^{i(t-s)H}\big(UB_j[(1-\Theta(t))z^\pm,z^\pm]
+U B_j[\Theta(t)z,(1-\Theta)(t)z^\pm]\big)ds\bigg\|_{W^{k,p}}\\
&\leq  \int^{t-1}_{1}\frac{1}{(t-s)^{1+3\ve}}\|z\|_{W^{k+2,p_1}}  \|(1-\Theta(s))z\|_{H^{k+2}}ds\\
&\leq  \int^{t-1}_{1}\frac{1}{(t-s)^{1+3\ve}}\|z\|_{H^N}^2\frac{1}{s^{\delta(N-2-k)}}ds,
\end{aligned}
\end{equation}
choosing $N$ large enough so that $\delta(N-2-k)\geq 1+3\ve$, we obtain the expected decay. 
\subsubsection*{Low frequencies}
Following the section $\ref{secNRT}$, we have to estimate quadratic term of the form $UB_j[z^\pm,z^\pm]$ wich leads to consider:
$${\cal F}I^{a,b,c,X}_3=e^{itH(\xi)}
\int_1^{t-1}\int_{\R^N} \bigg(( e^{-is\Omega} UB^{a,b,c,X}_j(\eta,\xi-\eta)
\widetilde{\Theta z^\pm}(s,\eta)\widetilde{\Theta z^\pm}(s,\xi-\eta)\bigg)d\eta ds,$$
with  $\Omega=H(\xi)\mp H(\eta)\mp H(\xi-\eta)$. Using  $\displaystyle e^{-is\Omega}=\frac{i\nabla_{\eta}\Omega}{s|\nabla_{\eta}\Omega|^2}
\cdot \nabla_{\eta}e^{-is\Omega}$ and denoting $Ri=\frac{\n}{|\n|}$ 
the Riesz operator, $\Theta'(t):=\theta'(\frac{|D|}{t^\delta})$, $J=e^{itH}xe^{-itH}$, an integration by 
part in $\eta$ gives:
\begin{equation}
\begin{aligned}
I^{a,b,c,X}_3=
&-{\cal F}^{-1}(e^{itH(\xi)}\biggl(
\int_1^{t-1}\frac{1}{s} \int_{\R^N}  \big(e^{-is \Omega(\xi,\eta)} \mathcal{B}^{a,b,c,X}_{1,j}(\eta,\xi-\eta)
\cdot \n_\eta [\Theta\widetilde{z^\pm}(\eta)\Theta\widetilde{z^\pm}(\xi-\eta)]\\
&\hspace{65mm}+\mathcal{B}^{a,b,c,X}_{2,j}(\eta,\xi-\eta)\widetilde{\Theta z^\pm}(\eta)
\widetilde{\Theta z^\pm}(\xi-\eta)d\eta\big)ds\biggl)\\
=&
-\int_1^{t-1}\frac{1}{s}e^{i(t-s)H}\bigg(\mathcal{B}^{a,b,c,X}_{1,j}
[\Theta(s)(Jz)^\pm,\Theta(s)z^\pm]
-\mathcal{B}^{a,b,c,X}_{1,j}[\Theta(s)z^\pm,\Theta(s)(Jz)^\pm]\\
&\hspace{8cm}
+\mathcal{B}^{a,b,c,X}_{2,j}[\Theta(s)z^\pm,\Theta(s)
z^\pm]\bigg)ds\\
&-\int_1^{t-1} \frac{1}{s}e^{i(t-s)H}\bigg(\mathcal{B}^{a,b,c,X}_{1,j}
[\frac{1}{s^\delta}Ri\,\Theta'(s)z^\pm,\Theta(s) z^\pm]\\
&\hspace{3.5cm}-\mathcal{B}^{a,b,c,X}_{1,j}[\Theta(s)z^\pm,\frac{1}{s^\delta}Ri
\Theta'(s)z^\pm]\bigg)ds.
\end{aligned}
\label{ippespace}
\end{equation}
 with:
\begin{equation*}
\displaystyle\mathcal{B}^{a,b,c,X}_{1,j}=\frac{ U(\xi)\nabla_{\eta}\Omega}{|\nabla_{\eta}
\Omega|^2}B^{a,b,c,X}_j,\
\displaystyle\mathcal{B}^{a,b,c,X}_{2,j}=\nabla_{\eta}B_{1,j}^{a,b,c,X}.
\end{equation*}
The following counterpart of lemma $\ref{bacrucial1}$ slightly improves the estimates from 
\cite{GNT3}.
\begin{lemma}
\label{bacrucial} Denoting $M=\max (a,b,c)$, $m=\min(a,b,c)$ and $l=\min(b,c)$ we have:
\begin{itemize}
\item If $M<<1$ then for $0\leq s\leq 2$: 
\begin{equation}
\|\mathcal{B}^{a,b,c,X}_{1,j}\|_{[B^{s}]}\lesssim l^{\frac{3}{2}-s}M^{1-s},\;\;
\|\mathcal{B}^{a,b,c,X}_{2,j}\|_{H^{s}}\lesssim l^{\frac{1}{2}-s}M^{-s},
\end{equation}
\item If $M\gtrsim 1$ then for $0\leq s\leq 2$:
\begin{equation}
\|\mathcal{B}^{a,b,c,X}_{1,j}\|_{[B^{s}]}\lesssim \la M\ra^2l^{3/2-s}\la a\ra^{-1},\;\;
\|\mathcal{B}^{a,b,c,X}_{2,j}\|_{[B^{s}]}\lesssim \la M\ra^2l^{1/2-s}\la a\ra^{-1},
\end{equation}
\end{itemize}
\end{lemma}
\noindent We now use these estimates to bound the first term of \eqref{ippespace}. Since 
they are independent of $j$ we now drop this index for concision. As in paragraph \ref{secNRT}
the $j$ index is dropped for conciseness, and 
there are 
three areas to consider: $b\lesssim c\sim a,\ c\lesssim c\lesssim a\sim b,\ a\lesssim b\sim c$.
\subparagraph{The case \mathversion{bold}{$c\lesssim a\sim b$}}  Let $\varepsilon_1>0$ to be 
fixed later. Using Minkowski's inequality, 
dispersion and the rough multiplier theorem \ref{singmult} with $s=1+\ve_1$, $\frac{1}{q}=1/2+\e-\frac{\e_1}{3}$
for $a\lesssim 1$, $s=4/3$, $\frac{1}{q_1}=7/18+\e$ for $a\gtrsim 1$ we obtain
$$
\begin{aligned}
&\big\|\n^{k_1}\int_1^{t-1}\frac{1}{s}e^{i(t-s)H}\sum_{c\lesssim a\sim b}
\mathcal{B}^{a,b,c,X}_{1}[\Theta(s)(Jz)^\pm,\Theta(s)z^\pm]ds\big\|_{L^{p}}\\
&\lesssim \int_1^{t-1}\frac{1}{s(t-s)^{1+3\ve}}\sum_{c\lesssim a\sim b\lesssim 1}
\|\mathcal{B}^{a,b,c,X}_{1}\|_{[B^{1+\ve_1}]}\|\Theta(s)Jz\|_{L^2}
\|\Theta(s)z]\|_{L^{q}}\\
&\hspace{13mm}+\sum_{c\lesssim a\sim b,\ 1\lesssim a\lesssim s^\delta}a^{k}
\|\mathcal{B}^{a,b,c,X}_{1}\|_{[B^{4/3}]}\|\Theta(s)Jz\|_{L^{2}}
\|\Theta(s)z]\|_{L^{q_1}}\big)ds\\[2mm]
&\lesssim \int_1^{t-1}\frac{1}{s(t-s)^{1+3\ve}}\big(\sum_{a\lesssim 1}\sum_{c\lesssim a\sim b}
\|\mathcal{B}^{a,b,a,X}_{1}\|_{[B^{1+\ve_1}]}\|\Theta(s)Jz\|_{L^2}
\|\Theta(s)z]\|_{L^{q}}\\
&\hspace{2cm}+\sum_{1\lesssim a \lesssim s^{\delta}}a^{k}\sum_{c\lesssim a\sim b}
\|\mathcal{B}^{a,b,c,X}_{1}\|_{[B^{4/3}]}\|\Theta(s)Jz\|_{L^{2}}
\|\Theta(s)z]\|_{L^{q_1}}\big)ds
\end{aligned}
$$
Using lemma \ref{bacrucial} and interpolation we have for $\ve_1<1/4$ and 
$\varepsilon_1-3\varepsilon>0$,
$$
\begin{aligned}
&\sum_{a\lesssim 1}\sum_{c\lesssim a\sim b}
\|\mathcal{B}^{a,b,c,X}_{1}\|_{[B^{1+\ve_1}]}\lesssim 
\sum_{a\lesssim 1}a^{1-(1+\ve_1)}\sum_{c\lesssim a}c^{\frac{3}{2}-(1+\ve_1)}\lesssim 1,\\
&\|\psi(s)\|_{L^{q}}\lesssim\|\psi(s)\|^{\frac{\e_1-3\e}{1+3\e}}_{L^{p}}\|\psi(s)\|^{1-\frac{\e_1-3\e}{1+3\e}}_{L^{2}} \lesssim \frac{\|\psi\|_{X}}{s^{\ve_1-3\ve}}.
\end{aligned}
$$ 
In high frequencies we have:
$$
\begin{aligned}
&\sum_{1\lesssim a \lesssim s^{\delta}}a^{k} \sum_{c\lesssim a\sim b} 
\frac{\la M\ra^2c^{3/2-4/3}}{\la a\ra}
\lesssim s^{\delta(k+7/6)},\
\|\psi(s)\|_{L^{q_1}}\lesssim \frac{\|\psi\|_X}{s^{1/3-3\varepsilon}}
\end{aligned}
$$
Finally we conclude that if $\min\big(\varepsilon_1-3\varepsilon,1/3-3\varepsilon
-\delta(k+7/6)\big)\geq 3\varepsilon$ (this choice is possible provided $\ve$ and $\delta$ are 
small enough):
$$
\begin{aligned}
\|\n^{k_1}\int_1^{t-1}\frac{1}{s}e^{-i(t-s)H}\big(\sum_{a,b,c}\mathcal{B}^{a,b,c,X}_{1}
[\Theta(s)(Jz)^\pm,\Theta(s)z^\pm]ds\|_{L^{p}}
&\lesssim \int_1^{t-1}\frac{\|z\|_X^2}{s^{1+3\ve}(t-s)^{1+3\ve}}
ds
\\
&\lesssim \frac{\|z\|_X^2}{t^{1+3\ve}}.
\end{aligned}
$$
The case $b\lesssim c\sim a$ is very similar, the case $a\lesssim b\sim c$ involves an infinite sum
over $a$ which can be handled as in the non time resonant case with observation \eqref{tricksuma}.
The term $\displaystyle \n^{k_1}\int_1^{t-1}\frac{1}{s}e^{i(t-s)H}\mathcal{B}^{a,b,c,X}_{1}
[\Theta(s)z^\pm,\Theta(s)(Jz)^\pm]ds$ is symmetric  while the terms
$$
\begin{aligned}
&\|\n^{k_1}\int_1^{t-1} \frac{1}{s}e^{i(t-s)H}\big(\mathcal{B}^{a,b,c,X}_{1}
[\frac{1}{s^\delta}Ri\Theta'(s)z^\pm,\Theta(s)z^\pm]\\
&\hspace{3cm}-\mathcal{B}^{a,b,c,X}_{1}[\Theta(s)z^\pm,
\frac{1}{s^\delta}Ri\Theta'(s)z^\pm]\big)ds\|_{L^p},
\end{aligned}
$$
are simpler since there is no weighted term $Jz$ involved.\vspace{2mm} \\
The last term to consider 
is 
$$\big\|\n^{k_1}\int_1^{t-1}\frac{1}{s}e^{i(t-s)H}\sum_{a,b,c}\mathcal{B}^{a,b,c,X}_{2}
[\Theta(s)z^\pm,\Theta(s)z^\pm]ds\big\|_{L^p}.  $$
Let us start with the zone $b\lesssim a\sim c$. We use the same indices as for 
$\mathcal{B}_1^{a,b,c}$: $s=1+\varepsilon_1$,
$\frac{1}{q}=1/2+\ve-\ve_1/3$, $s_1=4/3$, $\frac{1}{q_1}=7/18+\varepsilon$,
\begin{equation}\label{estimX}
\begin{aligned}
&\big\|\n^{k_1}\int_1^{t-1}\frac{1}{s}e^{i(t-s)H}\sum_{b\lesssim a}\mathcal{B}^{a,b,c,X}_{2}
[\Theta(s)z^\pm,\Theta(s)z^\pm]ds\big\|_{L^p}\\
&\lesssim \int_1^{t-1}\frac{1}{s(t-s)^{1+3\ve}}\big(\sum_{a\lesssim 1}\sum_{b\lesssim a\sim c}
U(b)U(c)\|\mathcal{B}^{a,b,c,X}_{2}\|_{[B^{1+\ve_1}]}\|U^{-1}\Theta(s)z\|_{L^2}\|U^{-1}\Theta(s)
z]\|_{L^q}\\
&\hspace{3cm}
+\sum_{1\lesssim a \lesssim s^{\delta}}a^{k}\sum_{b\lesssim a\sim c}\frac{U(b)}{\la c\ra^k}
\|\mathcal{B}^{a,b,c,X}_{2}\|_{[B^{4/3}]}\|U^{-1}\Theta(s)z\|_{L^{2}}
\|\la \n\ra^k\Theta(s)z]\|_{L^{q_1}}\big)ds
\end{aligned}
\end{equation}
For $M\lesssim 1$ we have if $\varepsilon_1<1/4$:
$$
 \sum_{a\lesssim 1}\sum_{b\lesssim c\sim a}U(b)U(c)\|\mathcal{B}_2^{a,b,c,X}\|
 _{[B^{1+\varepsilon_1}]}\lesssim 
 \sum_{a\lesssim 1}\sum_{b\lesssim c\sim a}b^{1/2-\varepsilon_1}a^{-\varepsilon_1}\lesssim 1.
$$ 
Furthermore we have from proposition $\ref{controlX}$:
\begin{equation*}
\|U^{-1}\psi(s)\|_{L^2}\lesssim\|\psi\|_{X},
\
\|U^{-1}\psi(s)\|_{L^{q}}\lesssim \|U^{-1}\psi\|_{L^2}^{1-\ve_1+3\ve}
\|U^{-1}\psi\|_{L^6}^{\ve_1-3\ve}\lesssim \frac{\|\psi\|_{X}}{s^{\frac{3(\ve_1-3\ve)}{5}}},
\end{equation*}
Now for $M\gtrsim 1$ $$
\sum_{1\lesssim a \lesssim s^{\delta}}a^{k}\sum_{b\lesssim c\sim a}
\frac{ U(b) \la M\ra^2b^{1/2-4/3}}{\la a\ra \la c\ra^k}\lesssim \sum_{1\lesssim a\lesssim s^\delta}
a \lesssim s^{\delta},\hspace{0.4cm} 
\|\la \nabla\ra^k\Theta(s)z\|_{L^{q_1}}\lesssim \frac{\|z\|_X}{s^{1/3-3\varepsilon}}.$$
If $\min\big(3(\varepsilon_1-3\varepsilon)/5,1/3-3\varepsilon-\delta\big)\gtrsim 3\varepsilon$,
injecting these estimates in \eqref{estimX} gives
$$
\big\|\n^{k_1}\int_1^{t-1}\frac{1}{s}e^{i(t-s)H}\big(\sum_{b\lesssim c\sim a}
\mathcal{B}^{a,b,c,X}_{2}
[\Theta(s)Jz,\Theta(s)z]ds\big\|_{L^{p}}
\lesssim \int_1^{t-1}\frac{\|z\|_X^2}{(t-s)^{1+3\ve}s^{1+3\ve}}ds
\lesssim \frac{\|z\|_X^2}{t^{1+3\varepsilon}}.
$$
The two other cases $c\lesssim a\sim b$ and $a\lesssim b\sim c$ can be treated in a similar way, we refer
again to the observation \eqref{tricksuma} in the case $a\lesssim b\sim c$.\\
It concludes this section, the combination of paragraphs \ref{secNRT} and \ref{secNRS} gives
\begin{equation*}
\bigg\|\int_0^te^{i(t-s)H}Q(z(s))ds\bigg\|_{W^{k,p}}\lesssim 
\frac{\|z\|_X^2+\|z\|_X^3}{\ct^{1+3\varepsilon}}.
\end{equation*}

\begin{rmq}\label{precN}
From the energy estimate, we recall that we need $k\geq 3$ (see \eqref{Fspaces}). The strongest 
condition on $N$ seems to be $(N-2-k)\delta> 1$. In the limit $\varepsilon\rightarrow 0$, we must 
have at least $1/3-\delta(k+7/6)>0$, so that $N\geq  18$.
\end{rmq}

\subsection{Bounds for the weighted norm }
The estimate for $\|x\int_0^te^{-isH}B_j[z,z]ds\|_{L^2}$ can be done with almost the same 
computations as in section $10$ from \cite{GNT3}. The only difference is that Gustafson et al
deal with nonlinearities without loss of derivatives. As we have seen in paragraph \ref{declp}, 
the remedy is to use appropriate frequency truncation, so we will only give a sketch 
of proof for the bound in this paragraph.
\paragraph{First reduction} 
Applying $xe^{-itH}$ to the generic bilinear term $U\circ B_j[z^\pm,z^\pm]$, we have for 
the Fourier transform:
\begin{equation}
\begin{aligned}
\mathcal{F}\big(xe^{-itH}\int_0^te^{i(t-s)H}UB_j[z^\pm,z^\pm]\big)=&
\int_0^t\int_{\R^d} \nabla_{\xi}\bigg(e^{-is\Omega}UB_j(\eta,\xi-\eta)
\widetilde{z^\pm}(s,\eta)\widetilde{z^\pm}(s,\xi-\eta)\biggl)d\eta \, ds\\
\end{aligned}
\end{equation}
As the $X_T$ norm only controls $\|Jz\|_{L^2}$, we have to deal with the loss of derivative in the nonlinearities. It is then convenient that $\xi-\eta\lesssim \eta$ in order to absorb the loss of derivatives;  to do this we use a cut-off function $\theta(\xi,\eta)$ 
which is valued in $[0,1]$, homogeneous of degree $0$, smooth outside of $(0,0)$ and such that 
$\theta(\xi,\eta)=0$ in a neighborhood of $\{\eta=0\}$ and $\theta(\xi,\eta)=1$  in a neighborhood 
of $\{\xi-\eta=0\}$ on the sphere. Using this splitting we get two terms
\begin{equation}
\begin{aligned}
&\int_0^t\int_{\R^d} \nabla_{\xi}\bigg(e^{-is\Omega}UB_j(\eta,\xi-\eta)\theta(\xi,\eta)
\widetilde{z^\pm}(s,\eta)\widetilde{z^\pm}(s,\xi-\eta)\biggl)d\eta \, ds, \\
&\int_0^t\int_{\R^d} \nabla_{\xi}\bigg(e^{-is\Omega}(1-\theta(\xi,\eta))UB_j(\eta,\xi-\eta)
\widetilde{z^\pm}(s,\eta)\widetilde{z^\pm}(s,\xi-\eta)\biggl)d\eta \, ds.
\end{aligned}
\label{intGer}
\end{equation}
By symmetry it suffices to consider the first one which corresponds to a region 
where $|\eta|\gtrsim |\xi|,|\xi-\eta|$ so that we avoid loss of derivatives for 
$\nabla_\xi \widetilde{z^\pm}(s,\xi-\eta)$.  
\paragraph{An estimate in a different space and high frequency losses}
Depending on which term  $\n_{\xi}$ lands, the following integrals arise:
$$
\begin{aligned}
{\cal F}I_1&=
\int_0^t\int_{\R^N} e^{-is\Omega}\n_\xi^{(\eta)} (\theta(\xi,\eta)UB_j(\eta,\xi-\eta))
\widetilde{z^\pm}(s,\eta)\widetilde{z^\pm}(s,\xi-\eta)d\eta ds,\\
{\cal F} I_2&=
\int_0^t\int_{\R^N} e^{-is\Omega} \theta(\xi,\eta)UB_j(\eta,\xi-\eta)
\widetilde{z^\pm}(s,\eta)\n_{\xi}^{(\eta)} \widetilde{z^\pm}(s,\xi-\eta)d\eta ds,\\
{\cal F} I_3&=
\int_0^t\int_{\R^N} e^{-is\Omega} (is\nabla_\xi\Omega)\theta(\xi,\eta)UB_j(\eta,\xi-\eta)
\widetilde{z^\pm}(s,\eta) \widetilde{z^\pm}(s,\xi-\eta)d\eta ds\\
&:= \mathcal{F}\bigg(\int_0^t e^{-isH} s\mathcal{B}_j[z^\pm,z^\pm]ds\bigg),
\end{aligned}
$$
with:
$$\mathcal{B}_j(\eta,\xi-\eta)=(is\nabla_\xi\Omega)\theta(\xi,\eta)UB_j(\eta,\xi-\eta).$$
The control of the $L^2$ norm of $I_1$ and $I_2$ is not a serious issue: basically we deal here with 
smooth multipliers, and from the estimate $\|z\,xe^{-itH}z\|_{L^1_TL^2}\lesssim 
\|z\|_{L^1_TL^\infty}\|xe^{-itH}z\|_{L^\infty _TL^2}\lesssim \|z\|_{X_T}^2$ it is apparent that we can conclude. The only point is that we can control the loss of derivative on $Jz$ via the truncation function $\theta_1$ and it suffices to absorb the loss of derivatives by $z$. 
Due to the $s$ factor, the case of $I_3$ is 
much more intricate and requires to use again the method of space-time resonances.\\
Let us set 
\begin{equation*}
\begin{aligned}
&\|z\|_{S_T}=\|z\|_{L^\infty_TH^1}+\|U^{-1/6}z\|_{L^2_T W^{1,6}},\\
&\|z\|_{W_T}=\|x e^{-itH}z\|_{L^\infty_T H^1}.
\end{aligned}
\end{equation*}
Gustafson et al prove in \cite{GNT3} the key estimate 
\begin{equation*}
\big\|\int_0^te^{-isH}sB[z^\pm,z^\pm]ds\big\|_{L^\infty_TL^2}\lesssim \|z\|_{S_T\cap W_T}^2,
\end{equation*}
where $B$ is a class of multipliers very similar to our $\mathcal{B}_j$, the only difference being 
that they are associated to semi-linear nonlinearities, and thus cause no loss of derivatives at 
high frequencies. We point out that the $S_T$ norm is weaker than the $X_T$ norm, indeed 
$\|U^{-1/6}z\|_{L^2_T W^{1,6}}\lesssim \|z\|_{L^2_TW^{2,9/2}}\lesssim \|z\|_{X_T}
\|1/\ct^{5/6}\|_{L^2_T}\lesssim \|z\|_{X_T}$. Moreover we have already seen how to deal with 
high frequency loss of derivatives by writing (see paragraph
\ref{secNRS})
\begin{equation}
\mathcal{B}_j[z^\pm,z^\pm]=\mathcal{B}_j[1-\Theta(t)z^\pm,z^\pm]+\mathcal{B}_j[\Theta(t)z^\pm,z^\pm].
\label{poidsterm}
\end{equation}
Let $1/q=1/3+\varepsilon$, the first term is estimated using Sobolev embedding and the fact that $N$ is large enough compared to $\delta$:
\begin{eqnarray*}
\big\|\int_0^t\int_{\R^N} e^{-isH} s\mathcal{B}_j[z^\pm,z^\pm]ds\big\|_{L^2}
\lesssim\int_0^t s\|(1-\Theta(s))z\|_{W^{3,q}}\|z\|_{W^{3,p}}ds
&\lesssim& \int_0^t\frac{\|z\|_{H^N}\|z\|_{X_T}}{\cs^{(N-4)\delta}}ds\\
&\lesssim& \|z\|_{X_T}^2.
\end{eqnarray*}
The estimate of the second term of (\ref{poidsterm}) follows from the (non trivial) computations
in \cite{GNT3}, section $10$. They are very similar to the analysis of the previous section (based 
on the method of space-time resonances), for the sake of completeness we reproduce hereafter 
a small excerpt from their computations.
\\
As in section \ref{declp}, one starts by splitting the phase space
\begin{equation*}
\int_0^t e^{i(t-s)H}s\mathcal{B}_j[\Theta(s) z^\pm,z^\pm]ds=\sum_{a,b,c}
\int_0^t e^{i(t-s)H}s\big(\mathcal{B}_j^{a,b,c,T}+\mathcal{B}_j^{a,b,c,X}\big)[\Theta(s)z^\pm,z^\pm]ds
\end{equation*}
For the time non-resonant terms, an integration by parts in $s$ implies:
\begin{equation}
\begin{array}{ll}
\displaystyle\int_0^t e^{i(t-s)H}s\mathcal{B}_j^{a,b,c,T}[\Theta(s)  z^\pm,z^\pm]ds\\
\displaystyle\hspace{3cm}=-\int_0^te^{isH}\bigg((\mathcal{B}'_j)^{a,b,c,T}[\Theta(s) z^\pm,z^\pm]ds
\displaystyle+(\mathcal{B}'_j)^{a,b,c,T}[s\Theta(s) \mathcal{N}_z^\pm,z^\pm]\\
\displaystyle\hspace{35mm}+(\mathcal{B}'_j)^{a,b,c,T}[\Theta(s)  z^\pm,s\mathcal{N}_z^\pm]+(\mathcal{B}'_j)^{a,b,c,T}[-\delta s^{-\delta}\Theta(s) |\n| z^\pm,z^\pm]\bigg)ds\\
\hspace{35mm}+
\displaystyle\big[e^{isH}(\mathcal{B}'_j)^{a,b,c,T}[s \Theta(s) z^\pm,z^\pm]\big]_0^t,
\end{array}
\label{10.41}
\end{equation}
with:
$$
(\mathcal{B}'_j)^{a,b,c,T}=\frac{1}{\Omega}\mathcal{B}_j^{a,b,c,T}=
\frac{i\nabla_{\xi}\Omega}{\Omega}B_j^{a,b,c,T}\theta(\xi,\eta),$$
We only consider the second term in the right hand side of \eqref{10.41}, in the case 
$c\lesssim b\sim a$. All the other terms can be treated in a similar way. The analog of lemma \ref{bacrucial1} in these settings is the following:
\begin{lemma}
Denoting $M=\max (a,b,c)$, $m=\min(a,b,c)$ and $l=\min(b,c)$ we have:
\begin{equation}
\|(\mathcal{B}_j')^{a,b,c,T}\|_{[H^{s}]}\lesssim \la M\ra^2\bigg(\frac{\la M \ra}{M}\bigg)^s 
l^{\frac{3}{2}-s}\la a\ra ^{-1}.
\end{equation}
\label{acrucial1}
\end{lemma}
We have then by applying theorem \ref{singmult}:
\begin{equation}
\begin{aligned}
\|\int^{T}_{0}e^{-isH} \sum_{c\lesssim a\sim b}(\mathcal{B}'_j)^{a,b,c,T}&[s  \Theta(s)\mathcal{N}z^\pm,z^\pm]
ds\|_{L^2}\\
&\lesssim \big\|\sum_{c\lesssim a\sim b} \frac{U(c)}{\la b\ra^{2}}\|(\mathcal{B}'_j)^{a,b,c,T}\|_{[B^{1+\varepsilon}]}
\| s \la \n\ra^2\mathcal{N}_z\|_{L^{2}} \| U^{-1}z\|_{L^{\infty}(L^{6})}\big\|_{L^1_T}
\end{aligned}
\label{10.44a}
\end{equation}
From lemma \ref{acrucial1} we find 
\begin{equation}
\begin{aligned}
\sum_{c\lesssim a\sim b}U(c)\|(\mathcal{B}_3')^{a,b,c,T}\|_{[B^1]}
&\lesssim \sum_{ c\lesssim a}\frac{U(c)}{\la a\ra^2}\la a \ra^2a^{-1}
c^{\frac{1}{2}},\\
&\lesssim \sum_{ a\leq 1}a^{1/2}+\sum_{a\geq 1}a^{-1/2}\lesssim 1.
\end{aligned}
\label{10.44ab}
\end{equation}
Next we have (as previously forgetting cubic and quartic nonlinearities) 
$$\|\la\n \ra^2\mathcal{N}_z\|_{L^2}\lesssim \|z\|_{W^{4,4}}^2\lesssim 
\|z\|_{X_T}^2/\la s\ra^{3/2},$$ and from \eqref{5.9}
$\|U^{-1}z(s)\|_{L^6}\lesssim \la s\ra^{-3/5}$ so that 
\begin{equation*}
\begin{aligned}
\|\int^{T}_{0}e^{-isH} \sum_{c\lesssim a\sim b}(\mathcal{B}'_j)^{a,b,c,T}&[s\mathcal{N}z^\pm,z^\pm]
ds\|_{L^2}\lesssim \|\|z\|_{X_T}^3\la s\ra^{-21/10}\|_{L^1_T}\lesssim \|z\|_{X_T}^3.
\end{aligned}
\end{equation*}

\subsection{Existence and uniqueness}
The global existence follows from the same argument as in dimension larger than $4$: for $N=3,4$ 
combining the energy estimate (proposition \ref{energy}), the a priori estimates for cubic, quartic 
(section \ref{estimR}) and quadratic nonlinearities (section \ref{estimquad}) and the proposition \ref{estimformenormale} we have uniformly in 
$T$
\begin{eqnarray*}
\begin{aligned}
&\|\psi\|_{X_T}\leq C_1\bigg(\|\psi_0\|_{W^{k,4/3}}+\|\psi_0\|_{H^N}+\|\psi\|_{X_T}^2 G(\|\psi\|_{X_T},\|\frac{1}{1+l_1}\|_{L^\infty_T(L^\infty)})\\
&\hspace{3cm}+ \|\psi_{0}\|_{H^{2n+1}}\text{exp}\big(C'\|\psi\|_{X_T} H(\|\psi\|_{X_T},\|\frac{1}{l+1}\|_{L^\infty_T(L^\infty)})\big)\bigg).
\end{aligned}
\end{eqnarray*}
with $G$ and $H$ continuous functions so that from the standard bootstrap argument and the blow up criterion (see page 
\pageref{blowcriter}) the local solution is global.
\subsection{Scattering}\label{secscatt}
It remains to prove that $e^{-itH}\psi(t)$ converges in $H^s(\R^3)$, $s< 2n+1$. This is a 
consequence of the following lemma:
\begin{lemma}
For any $0\leq t_1\leq t_2$, we have 
\begin{equation}\label{estimsanspoids}
\|\int_{t_1}^{t_2} e^{isH}\mathcal{N}\psi ds\|_{L^2}\lesssim \frac{\|\psi\|_{X}^2}{(t_1+1)^{1/2}}.
\end{equation}
\end{lemma}
\begin{proof}
We focus on the quadratic terms since the cubic  and quartic terms give even stronger decay.
From Minkowski and H\"older's inequality and the dispersion 
$\|\psi\|_{L^{p}}\leq \frac{\|\psi\|_X}{\ct+^{3(1/2-1/p)}}$: 
\begin{eqnarray*}
\|\int_{t_1}^{t_2}e^{-i(t-s)H}{\cal N}\psi ds\|_{L^2} 
\lesssim \int_{t_1}^{t_2}\|\la\n\ra^2\psi\la\n\ra^2\psi \|_{L^2}ds,
& \lesssim &\int_{t_1}^{t_2}\|\la\n\ra^2\psi \|_{L^4}^2ds,\\
 &\lesssim& \|\psi\|_{X}^2\int_{t_1}^{t_2}\frac{1}{\cs^{d/2}}ds .
\end{eqnarray*}
\end{proof}
\noindent Interpolating between the uniform bound in $H^{2n+1}$ and the decay in $L^2$ we get 
$$\|e^{-it_1H}\psi(t_1)-e^{-it_2H}\psi(t_2)\|_{H^s}\lesssim 1/\la t_1\ra^{(2n+1-s)/(4n+2)},$$ 
thus 
$e^{-itH}u$ converges in $H^s$ for any $s<2n+1$. For $d=3$, the convergence of 
$xe^{-itH}\psi$ in $L^2$ follows from an elementary but cumbersome inspection of the proof of 
boundedness of $xe^{-itH}\psi$. If one replaces everywhere 
$\displaystyle \int_0^txe^{-isH}\mathcal{N}_zds$ by 
$\displaystyle \int_{t_1}^{t_2}xe^{-isH}\mathcal{N}_zds$, every estimates ends up 
with $\|\psi\|_X^2\int_{t_1}^{t_2}/(1+s)^{1+\varepsilon'}ds$, $k=2,3,4,\ \varepsilon'>0$, 
so that $xe^{-itH}\psi$ is a Cauchy sequence in $L^2$. A careful inspection of the proof would 
also allow to quantify the value of $\varepsilon'$.

\appendix
\section{The multiplier estimates}
The aim of this section is to provide a brief sketch of proof of lemmas \ref{bacrucial} and 
\ref{bacrucial1}, let us recall that $\mathcal{B}_{1}$, $\mathcal{B}_{2}$ and $\mathcal{B}_{3}$ 
depend on the phase $\Omega=H(\xi)\mp H(\eta)\mp H(\xi-\eta)$ in the following way 
$$
\begin{aligned}
&\mathcal{B}^{a,b,c,T}_{3}=\frac{B_j}{\Omega}U(\xi)\chi^{a}\chi^b\chi^{c},\\
&\displaystyle\mathcal{B}^{a,b,c,X}_{1,j}=\frac{ B_j\nabla_{\eta}\Omega}{|\nabla_{\eta}
\Omega|^2}U(\xi)\chi^{a}(\xi)\chi^b (\eta)\chi^c (\xi-\eta),\\
&\displaystyle\mathcal{B}^{a,b,c,X}_{2,j}=\nabla_{\eta}\bigg(B_j\frac{ \nabla_{\eta}\Omega}
{|\nabla_{\eta}\Omega|^2}U(\xi)\chi^{a}(\xi)\chi^b (\eta)\chi^c (\xi-\eta)\bigg),
\end{aligned}
$$
Recall the notations:
\begin{equation}
\begin{aligned}
&|\xi|\sim a,\;|\eta|\sim b,\;|\zeta|\sim c,\\
&M=\max(a,b,c),\;m=\min(a,b,c),\;l=\min(b,c).
\end{aligned}
\label{11.6}
\end{equation}
The function $\chi_a$, resp. $\chi_b,\chi_c$, are smooth cut-off functions that localize 
near $|\xi|\sim a$ (resp $|\eta|\sim b,\ |\zeta|\sim c$).
We set as in \cite{GNT3}:
\begin{equation}
\alpha=|\hat{\zeta}-\hat{\xi}|,\;\beta=|\hat{\zeta}+\hat{\eta}|,\;
\eta^{\perp}=\hat{\xi}\times\eta.
\label{11.17}
\end{equation}
As a first reduction, we point out that the $B_j's$ satisfy the pointwise estimate 
\begin{equation}
\nabla^kB_j(\eta,\xi-\eta)|\lesssim \la M\ra ^{2}l^{-k}\label{estimBj} 
\end{equation}
We will see that the term $l^{-k}$ causes less loss of derivatives than if $\nabla_\eta$ hits 
$1/\Omega$ and $|\nabla_\eta\Omega|$, so that it will be sufficient 
to derive pointwise estimates for $\nabla^k(U/\Omega)$, $\nabla^k(U\nabla_\eta\Omega/|\nabla_\eta 
\Omega|^2$, and then multiply them by $\la M^2\ra$ to obtain pointwise estimates for the full 
multiplier.
\subsection{The \texorpdfstring{$+-$}{barzz} case}
If $\Omega=H(\xi)+H(\eta)-H(\xi-\eta)$ 
Gustafson et al in \cite{GNT3} decompose the $(\xi,\eta,\zeta)$ region (with $\zeta=\xi-\eta$) 
into the following five cases where each later case excludes the previous ones:
\begin{enumerate}
\item $|\eta|\sim|\xi|>>|\zeta|$ (or $c<<b\sim a$) temporally non-resonant.
\item $\alpha>\sqrt{3}$ temporally non-resonant.
\item $|\zeta|\geq 1$ spatially non-resonant.
\item $|\eta^{\perp}|<<M|\eta|$ temporally non resonant.
\item Otherwise spatially non-resonant.
\end{enumerate}
The estimates of lemmas $\ref{bacrucial}\ref{bacrucial1}$ are essentially a 
consequence of the pointwise 
estimates in \cite{GNT3}, section 11, except in the fifth case where we provide a necessary 
improvement. We sketch all five cases for completeness, 

\begin{enumerate}
\item If $|\eta|\sim|\xi|>>|\zeta|$, we have
\begin{equation}
|\Omega|=\Omega=H(\xi)+H(\eta)-H(\zeta)\geq H(M)\sim M\la M\ra.
\label{11.8a}
\end{equation}
\begin{equation}
|\n_{\zeta}\Omega|\lesssim |\n H(\eta)|\lesssim \la M\ra,
|\n^2_{\zeta}\Omega| \lesssim  \frac{\la m\ra}{m}.
\end{equation}
From these estimates, the $B_j$ estimate \eqref{estimBj}, the volume bound 
$|\{|\zeta|\sim m\}|\sim m^{3}$ and an interpolation 
argument we obtain $\displaystyle \big\|\frac{U(\xi)B_j}{\Omega}\chi^{a}\chi^{b}\chi^{c}\big\|
_{L^{\infty}_\xi(\dot{H}^{s}_{\zeta})}\lesssim m^{\frac{3}{2}-s}$, 
which is better than (\ref{10.43}).
\item In the second case $\alpha>\sqrt{3}$ so that $|\zeta|\sim |\eta|\gtrsim|\xi|$.\\
We cut-off the multipliers by:
$\displaystyle \chi_{[\alpha]}=\Gamma(\hat{\xi}-\hat{\zeta})$,
for a fixed $\Gamma\in C^{\infty}(\R^3)$ satisfying $\Gamma(x)=1$ for 
$|x|\geq\sqrt{3}$ and $\Gamma(x)=0$ for $|x|\leq\frac{3}{2}$. 
In this region, 
\begin{equation}
|\Omega |\geq \la M\ra |\xi| \sim  \la M\ra m,
|\n_\eta\Omega|\lesssim\frac{M m}{\la M\ra}+\frac{\la M\ra m}{M}\lesssim\frac{|\Omega|}{M},
\label{11.14}
\end{equation}
\begin{equation}
|\n_\eta^2\Omega|=|\n^2 H(\eta)-\n^2 H(\zeta)|=|\n^2 H(\eta)-\n^2 H(-\zeta)|
\lesssim\frac{\la M\ra m}{M^2}\lesssim\frac{|\Omega|}{M^2}.
\label{11.17a}
\end{equation}
As a consequence:
\begin{equation}
\|\frac{U(\xi)}{\Omega}\chi_{[\alpha|}\chi^{a}\chi^b\chi^{c}\|_{L^{\infty}_\xi(\dot{H}^{s}_{\eta})}\lesssim
\frac{\la M\ra^2}{m \la M\ra}\frac{M^{\frac{3}{2}}}{M^s}\frac{m}{\la m\ra}
=\frac{ \la M\ra M^{\frac{3}{2}-s}}{\la m\ra}\sim\frac{\la M\ra l^{\frac{3}{2}-s}}{\la a \ra}.
\label{11.18}
\end{equation}
\begin{remarka}
The use of the normal form is essential here as for general $B_j^{a,b,c}$ we would obtain in 
equation (\ref{11.18}):
\begin{equation}
\|\frac{U(\xi)}{\Omega}\chi_{[\alpha|}\chi^{a}\chi^b\chi^{c}\|_{L^{\infty}_\xi(\dot{H}^{s}_{\eta})}
\lesssim \frac{b^{3/2}}{m\la M\ra M^s\la m\ra }
\end{equation}
and the term $\frac{1}{m}$ could not be controlled. The same issue applies for the next areas.
\end{remarka}
\item  The case $M\sim|\zeta| \gtrsim 1$ and $\alpha<\sqrt{3}$. 
We remind that the symbols to estimate are:
\begin{equation}
\frac{\n_\eta\Omega}{|\n_\eta\Omega|^2 } U(\xi)
\chi^a(\xi)\chi^b(\eta)\chi^c(\xi-\eta),
\n_\eta\cdot((\mathcal{B}_{1}^{a,b,c,X})')
\label{utile}
\end{equation}
According to \cite{GNT3}, the pointwise estimates in this region are 
\begin{equation}
|\n_\eta\Omega|\sim ||\zeta|-|\eta||+\la\eta\ra \beta \gtrsim |\xi|,\
|\n^k_\eta \Omega|\lesssim\frac{\la\zeta\ra}{|\zeta|}|\xi|\,|\eta|^{1-k}\lesssim 
|\xi|\,|\eta|^{1-k}.
\label{11.21}
\end{equation}
Differentiating causes the same growth near $|\eta|=0$ as in \eqref{estimBj}, we deduce for 
$s\in[0,2]$
\begin{equation}
\begin{aligned}
&\big\|B_j\frac{\n_\eta\Omega}{|\n_\eta\Omega|^2}\chi^{C}_{[\alpha]}
U(\xi)\chi^a(\xi)\chi^b(\eta)\chi^c(\xi-\eta)\big\|_{\dot{H}^{s}_\eta}
\lesssim\frac{\la M\ra^2 b^{\frac{3}{2}}}{a b^s}U(a)=\la M\ra^2l^{\frac{3}{2}-s}\la a\ra^{-1},\\
&\big\|\n_\eta\cdot\big(\frac{\n_\eta\Omega}{|\n_\eta\Omega|^2}\cdot B_j\chi^{C}_{[\alpha]}
U(\xi)\chi^a(\xi)\chi^b(\eta)\chi^c(\xi-\eta)\big)\big\|_{\dot{H}^{s}_\eta}
\lesssim l^{\frac{1}{2}-s}\la a\ra^{-1}.
\end{aligned}
\label{11.23}
\end{equation}

\item The case $|\eta^{\ort}|<<M|\eta|$ corresponds to a low frequency region, where the symbol has 
a ``wave-like'' behaviour.  In this region
\begin{equation}
1>>M\sim|\zeta|,\;\alpha<\sqrt{3},\ |\eta^{\ort}|=|\eta||\sin(\widehat{(\eta,\xi)})|<<M |\eta|,
\label{11.26}
\end{equation}
The localization uses the (singular) cut-off multiplier 
$\displaystyle \chi_{[\ort]}=\chi\bigg(\frac{|\eta^\perp|}{100Mb}\bigg)$
with $\chi\in C^\infty_0(\R)$ satisfying $\chi(u)=1$ for $|u|\leq 1$ and $\chi(u)=0$ for 
$|u|\geq 2$. 
In particular 
$\displaystyle |\n^k_\eta\chi_{[\ort]}|\lesssim \big(\frac{1}{Mb}\big)^k,$
for all $k\geq 1$. 
The worst case is $M=|\zeta|$, in this case $\Omega$ does not cancel thanks to the slight 
radial convexity of $H$:
\begin{equation}
\Omega=H(\xi+\eta)-H(\xi)-H(\eta)\sim \frac{|\xi||\eta|(|\xi|+|\eta|)}{\la\xi\ra+\la\eta\ra}
\sim M^2m,\ 
|\n_\eta\Omega|<<|\xi|.
\label{11.33}
\end{equation}
For higher derivatives we have:
\begin{equation}
|\n^{1+k}_\eta\Omega|=|\n^{k+1}H(\eta)-\nabla^{k+1}H(\zeta)|\lesssim \frac{|\xi|}{M|\eta|^k},
\ |\nabla_\eta^kB_j|\lesssim l^{-k}.
\label{11.35}
\end{equation}
For $|\eta|\sim b,\ |\eta^\perp|<<Mb$, the region has for volume bound $b(Mb)^2=M^2b^3$,
we get by integration (for $s$ integer) and interpolation
\begin{equation}
\bigg\|\frac{U(\xi)}{\Omega}\chi_{[\ort]} \chi_{[\alpha]}^C  \chi^{a}\chi^b \chi^{c}
\bigg\|_{L^2_{\eta}}\lesssim \frac{U(a) (M^2 b^{3})^{1/2}}{M^2 m (Mb)^s }\lesssim 
l^{\frac{1}{2}-s}M^{-s} .
\label{11.36}
\end{equation}
\item 
In the last case we need a slight refinement of the symbol estimates from 
\cite{GNT3}: in the fifth area,
$|\eta^\perp|\gtrsim Mb\sim |\zeta||\eta|$, $M\sim |\zeta|<<1$, 
$\alpha=|\widehat{\zeta}-\widehat{\xi}|\leq \sqrt{3}$.\vspace{3mm}\\
We have $|\nabla_\eta\Omega|=|H'(|\eta|)\widehat{\eta}+H'(|\zeta|)\widehat{\zeta}|\sim 
H'(|\eta|)-H'(|\zeta|)+|\widehat{\eta}+\widehat{\zeta}|\geq |\widehat{\eta}+\widehat{\zeta}|$, and
$$
|\widehat{\eta}+\widehat{\zeta}|\geq \frac{|\eta\wedge \zeta|}{|\eta||\zeta|}
=\frac{|\eta\wedge (\xi-\eta)|}{|\eta||\zeta|}
=\frac{|\eta\wedge \xi|}{|\eta||\zeta|}=\frac{|\eta^\perp||\xi|}{|\eta||\zeta|}.
$$
indeed, if $\eta,\zeta$ form an angle $\theta$, 
$|\eta\wedge \zeta|=|\eta||\zeta|
|\sin\theta|\text{ and }|\widehat{\eta}+\widehat{\zeta}|\geq |\sin \theta)|.
$
Thus $|\nabla_\eta\Omega|\gtrsim |\xi||\eta^\perp|/(|\eta||\zeta|)\gtrsim |\xi|$ 
(in \cite{GNT3}, the authors only used $|\nabla_\eta\Omega|\gtrsim |\zeta|\,|\xi|$).\\
For the higher derivatives, we combine $(\ref{11.35})$ with 
$|\nabla_\eta\Omega|\gtrsim |\xi||\eta^\perp|/|\eta||\zeta|$ to get
\begin{equation}\forall\,k\geq 2,
\frac{|\n_\eta^k\Omega|}{|\n_\eta\Omega|}\lesssim\frac{|\xi|}{M|\eta|^{k-1}\beta}
\lesssim \frac{1}{|\eta|^{k-2}|\eta^\ort|}.
\end{equation}
so that we have the pointwise estimate
$$
\bigg|\nabla_\eta^k\frac{\nabla_\eta\Omega}{|\nabla_\eta\Omega|^2}\bigg|\sim 
\frac{1}{|\nabla\Omega|}\bigg(\frac{|\nabla^2_\eta\Omega|}
{|\nabla_\eta\Omega|}\bigg)^k\lesssim \frac{1}{|\xi||\eta^\perp|^{k}}.
$$
Following \cite{GNT3}, we then use a dyadic decomposition 
$|\eta^\perp|\sim \mu\in 2^j\Z,\ Mb\lesssim \mu\lesssim b.$
For each $\mu$ integrating gives a volume bound $\mu b^{1/2}$ and using interpolation we get
for $s>1$
\begin{equation*}
\|U(\xi)/|\nabla_\eta\Omega|\|_{\dot{H}^s_\eta}\lesssim \sum_{Mb\lesssim \mu\lesssim b}
\frac{U(a)\mu b^{1/2}}{a\mu ^{s}}\sim l^{3/2-s}M^{1-s}
\end{equation*}

\end{enumerate}

\subsection{The other cases}
\paragraph{The $-+$ case} This case is clearly symmetric from the $+-$ case.
\paragraph{The $--$ case}
The decomposition follows the same line as in \cite{GNT3}. Note however that the analysis 
is simpler at least for $M\geq 1$. Indeed in this area 
$|\n_\eta \Omega|\sim |H'(\eta)-H'(\zeta)|+|\widehat{\eta}-\widehat{\zeta}|\gtrsim 
\big||\eta|-|\zeta|\big|+|\widehat{\eta}-\widehat{\zeta}|\sim |\eta-\zeta|$
so that we might split it as $\{|\eta-\zeta|\gtrsim \max(|\eta|,|\zeta|)\}$ and 
$\{|\eta-\zeta|<<\max(|\eta|,|\zeta|)\}$. The first region is obviously space non resonant. 
The second region is 
time non resonant, indeed since $M\gtrsim 1$ we have in this region 
$|\xi|\sim |\eta|\sim |\zeta|\gtrsim 1$. Using a Taylor development gives
\begin{equation*}
H(\xi)-H(\eta)-H(\zeta)=
H(2\eta+\zeta-\eta)-H(\eta)-H(\eta+\zeta-\eta)=H(2\eta)-2H(\eta)+O(\la a\ra|\zeta-\eta|),
\end{equation*}
this last quantity is bounded from below by $|\eta|^2$ for 
$|\eta|\gtrsim 1$, $|\zeta-\eta|$ small enough.\\
For $M<1$, we can follow the same line as for $Z\overline{Z}$ by inverting the role of $\xi$ and 
$\zeta$. Note that the improved estimate in the last area relied on $|\n_\eta\Omega_{+-}|\gtrsim 
|\widehat{\eta}+\widehat{\zeta}|\geq |\eta^\perp|\xi|/(|\eta||\zeta|)$ and can just be replaced 
by $|\n_\eta\Omega_{--}|\gtrsim 
|\widehat{\eta}-\widehat{\zeta}|\geq |\eta^\perp|\xi|/(|\eta||\zeta|)$.
\paragraph{The $++$ case} We have $\Omega=H(\xi)+H(\eta)+H(\zeta)\gtrsim 
(|\xi|+|\eta|+|\zeta|)(1+|\xi|+|\eta|+|\zeta|)$, the area is time non resonant.
\subsection*{Acknowledgement:} The first author has been partially funded by the ANR project 
BoND ANR-13-BS01-0009-01. The second author  has been partially funded by the ANR project 
INFAMIE ANR-15-CE40-0011.

\bibliography{biblio}
\bibliographystyle{plain}

\end{document}